\documentclass[11pt,a4paper,twoside]{amsart}
\usepackage{amsmath,amsfonts,amsthm,amsopn,color,amssymb,enumitem}
\usepackage{palatino}
\usepackage{graphicx}
\usepackage[colorlinks=true]{hyperref}
\hypersetup{urlcolor=blue,citecolor=red,linkcolor=blue}

\usepackage{relsize}
\usepackage{esint}

\newcommand\e\varepsilon

\newcommand\R{\mathbb R}
\newcommand\Z{\mathbb Z}

\newcommand\de\partial
\newcommand\weakto\rightharpoonup

\DeclareMathOperator\dist{dist}

\DeclareMathOperator\supp{supp}

\renewcommand\le\leqslant
\renewcommand\ge\geqslant
\renewcommand\a\alpha

\renewcommand\b\beta

\renewcommand\d\delta
\newcommand\vfi\varphi
\newcommand\g\gamma
\newcommand\gb\gamma
\renewcommand\l\lambda
\newcommand\n\nabla
\newcommand\s\sigma
\renewcommand\t\theta
\renewcommand\O\S

\newcommand\G\Gamma

\renewcommand\S\Sigma
\renewcommand\L\Lambda

\newcommand\I{\mathcal I}

\newcommand\M{\mathcal M}
\newcommand\N{\mathbb N}

\renewcommand\o\S

\def\bbm[#1]{\text{\boldmath $#1$}}
\newcommand\beq{\begin{equation}}
\newcommand\eeq{\end{equation}}

\renewcommand\leq{\leqslant}

\newtheorem{theorem}{Theorem}[section]
\newtheorem{lemma}[theorem]{Lemma}

\newtheorem{proposition}[theorem]{Proposition}
\newtheorem{remark}[theorem]{Remark}
\newtheorem{corollary}[theorem]{Corollary}

\def\sideremark#1{\ifvmode\leavevmode\fi\vadjust{\vbox to0pt{\vss
\hbox to0pt{\hskip\hsize\hskip1em
\vbox{\hsize3cm\tiny\raggedright\pretolerance10000
\noindent #1\hfill}\hss}\vbox to8pt{\vfil}\vss}}}

\title[A double mean field equation related to a curvature prescription problem]
{A double mean field equation related to a curvature prescription problem}

\author{Luca Battaglia}
\address{Luca Battaglia\\ 
Universit\`a degli Studi Roma Tre\\
Dipartimento di Matematica e Fisica\\
Largo S. Leonardo Murialdo 1\\
00146 Roma, Italy.}
\email{lbattaglia@mat.uniroma3.it}

\author{Rafael L\'opez--Soriano}
\address{Rafael L\'opez--Soriano\\
Universitat de Val\`encia\\
Departamento de An\'alisis Matem\'atico\\
Dr. Moliner 50\\
46100 Burjassot (Val\`encia), Spain.}
\email{rafael.lopez-soriano@uv.es}



\everymath{\displaystyle}

\begin{document}

\begin{abstract}
We study a double mean field--type PDE related to a prescribed curvature problem on compacts surfaces with boundary:

\beq\label{meanmeanfield}
\left\{\begin{array}{ll}
-\Delta u=2\rho\left(\frac{Ke^u}{\int_\S Ke^u}-\frac1{|\S|}\right)&\text{in }\S,\vspace{.3cm}\\
\de_\nu u=2\rho'\left(\frac{he^\frac u2}{\int_{\de\S}he^\frac u2}-\frac1{|\de\S|}\right)&\text{on }\de\S.
\end{array}\right.
\eeq

Here $\rho$ and $\rho'$ are real parameters, $K,h$ are smooth positive functions on $\S$ and $\de\S$ respectively and $\nu$ is the outward unit normal vector to $\partial \Sigma$.

We provide a general blow--up analysis, then a Moser--Trudinger inequality, which gives energy--minimizing solutions for some range of parameters. Finally, we provide existence of min--max solutions for a wider range of parameters, which is dense in the plane if $\S$ is not simply connected.
\end{abstract}

\maketitle

\section{Introduction}
\setcounter{equation}0

Let $\S$ be a compact surface with boundary equipped with a metric $\widetilde g$. We consider the following boundary value problem

\beq\label{gg0}
\left\{\begin{array}{ll}
-\Delta u+2\widetilde K=2Ke^u&\text{in }\S,\vspace{.3cm}\\
\de_\nu u+2\widetilde h=2he^\frac u2&\text{on }\de\S,
\end{array}\right.
\eeq
where $\Delta=\Delta_{\widetilde g}$ is the Laplace--Beltrami operator in $(\S,\widetilde g)$, $\de_\nu$ is the normal derivative with $\nu$ the outward normal vector to $\de\S$ and $K,\widetilde K:\S\to\R$ and $h,\widetilde h:\de\S\to\R$ are smooth.

This kind of equations has a special interest due to its geometric meaning. Indeed, the problem allows us to prescribe at the same time Gaussian curvature in $\S$ and geodesic curvature on $\de\S$. More precisely, given a metric $g=\widetilde ge^u$ conformal to $\widetilde g$, if $K,\widetilde K$ are the Gaussian curvatures and $h,\widetilde h$ the geodesic curvatures of $\de\S$, relative to the metrics $g,\widetilde g$, then $u$ satisfies \eqref{gg0}.

Integrating \eqref{gg0} and applying the Gauss--Bonnet theorem, one obtains

$$\int_\S Ke^u+\int_{\de\S}he^\frac u2=\int_\S\widetilde K+\int_{\de\S}\widetilde h=2\pi\chi(\S),$$
which imposes necessary conditions on the choice of the functions $K,h$.

\

Some versions of the problem has been studied in the literature. Regarding the solvability, in \cite{chyg,ruizsoriano} the case $h=0$ is considered, whereas if $K=0$ there are some results available in \cite{chang2,liliu,liuhuang}.

Concerning the prescription of constant curvatures, it is worth referring to \cite{brendle}, where solutions are obtained by the use of a parabolic flow. By means of complex analysis, explicit solutions were found for the disk and the annulus, \cite{hangwang,asun}. There exist also some classification results when $\S$ is the half--plane in \cite{zhang,mira-galvez}. 

The case of non constant curvature has not been as much studied. For instance, \cite{cherrier} gives partial existence results, which includes an undetermined Lagrange multiplier; in \cite{hamza}, the author derives a Kazdan--Warner condition for the existence of solution.

\

It is easy to see that, using a conformal change of metric, we can always prescribe the constant values $\widetilde h\equiv0$, $\widetilde K\equiv\widetilde K_0:=\frac{2\pi\chi(\S)}{|\S|}$ (see \cite{lsmr}, Proposition 3.1, for a precise deduction). Hence, without loss of generality, we can assume that the initial metric satisfies $\widetilde K=\frac{2\pi\chi(\S)}{|\S|}$ and $\widetilde h=0$. The problem then becomes:

\beq\label{gg}
\left\{\begin{array}{ll}
-\Delta u+\frac{4\pi\chi(\S)}{|\S|}= 2Ke^u&\text{in }\S,\vspace{.3cm}\\
\de_\nu u=2he^\frac u2&\text{on
}\de\S.
\end{array}\right. 
\eeq

One possible strategy to obtain solutions to \eqref{gg} is to exploit the variational structure of the problem. In \cite{lsmr} the energy functional $\I:H^1(\S)\to\R$ is considered:

$$\I(u)=\int_\S\left(\frac12|\n u|^2+\frac{4\pi\chi(\S)}{|\S|}u-2Ke^u\right)-4\int_{\de\S}he^\frac u2.$$

By minimizing the previous Euler--Lagrange energy functional or via min--max methods, the authors obtain several existence results for surfaces with $\chi(\S)\le0$ and a compactness criterion for solutions. Actually, it is shown that the relation between $K$ and $h$ on $\de\S$ dramatically affects the geometry of $\I$.

\

An alternative variational formulation was introduced by Cruz and Ruiz in \cite{sergio}. Defining the parameter $\rho:=\int_\S Ke^u=2\pi\chi(\S)-\int_{\de\S}he^\frac u2$, the problem \eqref{gg} is equivalent to the following mean field equation:

\beq\label{meanfieldgg}
\left\{\begin{array}{ll}
-\Delta u+\frac{4\pi\chi(\S)}{|\S|}=2\rho\frac{Ke^u}{\int_\S Ke^u}&\text{in }\S,\vspace{.3cm}\\
\de_\nu u=2(2\pi\chi(\S)-\rho)\frac{he^\frac u2}{\int_{\de\S}he^\frac u2}&\text{on }\de\S,\vspace{.3cm}\\
\frac{(2\pi\chi(\S)-\rho)^2}{|\rho|}=\frac{\left(\int_{\de\S}he^\frac u2\right)^2}{\left|\int_\S Ke^u\right|}.
\end{array}\right.
\eeq

Solutions of the problem \eqref{meanfieldgg} can be found as critical points of this new energy functional, defined on $H^1(\S)\times\mathbb R$:

\begin{eqnarray*}
\I_0(u,\rho)&=&\int_\S\left(\frac12|\n u|^2+\frac{4\pi\chi(\S)}{|\S|}u\right)-2\rho\log\left|\int_\S Ke^u\right|-4(2\pi\chi(\S)-\rho)\log\left|\int_{\de\S}he^\frac u2\right|\\
&+&4(2\pi\chi(\S)-\rho)\log|2\pi\chi(\S)-\rho|+2\rho+2\rho\log|\rho|.
\end{eqnarray*}

In \cite{sergio}, the authors are concerned with the case in which $\S$ is the unit disk and the functions $K,h$ are nonnegative and verify certain symmetry properties. In this case $0<\rho<2\pi$ and $\I_0(u,\rho)$ is coercive by using a Moser--Trudinger type inequality, hence a solution can be derived by minimizing.

In this context, the mean field formulation seems to be convenient, but it does not seem to be as useful for noncoercive ranges.

\

Mean field type equations have been object of several works, see for instance \cite{dm,bdm,bjmr}, not only motivated by its geometrical meaning but also due to its relevance in some current physical theories. As a limit problem, they appear in the abelian Chern--Simons Theory (see \cite{dun}), or in the study of vortex type configurations in the Electroweak theory of Glashow--Salam--Weinberg, see \cite{lai}. We refer to the reader the monographs \cite{yang,tarbook} for further details and a complete set of references concerning these applications. 

\

In this paper we will focus on a slightly different case, namely

\beq\label{meanfield}
\left\{\begin{array}{ll}
-\Delta u+\frac{2(\rho+\rho')}{|\S|}=2\rho\frac{Ke^u}{\int_\S Ke^u}&\text{in }\S,\vspace{.3cm}\\
\de_\nu u=2\rho'\frac{he^\frac u2}{\int_{\de\S}he^\frac u2}&\text{on }\de\S.
\end{array}\right.
\eeq

Clearly, problem \eqref{meanfield} can be obtained by \eqref{gg} by setting $\rho:=\int_\S Ke^u$ and $\rho':=\int_{\de\S}he^\frac u2$, and includes also the \emph{nongeometrical} case $\rho+\rho'\ne2\pi\chi(\S)$. Anyway, unlike \eqref{meanfieldgg}, not all the solutions to \eqref{meanfield} correspond to \eqref{gg}, because the third relation in \eqref{meanfieldgg} may not be verified.

Moreover, it can be seen that equation \eqref{meanfield} can be transformed into \eqref{meanmeanfield}. In order to do it, it suffices to notice that $u+\vfi$ is solution to \eqref{meanmeanfield}, where $u$ solves \eqref{meanfield} and $\vfi$ satisfies

\beq\label{phi}
\left\{\begin{array}{ll}
\Delta\vfi=\frac{2\rho'}{|\S|}&\text{in }\S,\vspace{.3cm}\\
\de_\nu\vfi=\frac{2\rho'}{|\de\S|}&\text{on }\de\S.
\end{array}\right.
\eeq

Since the difference between both formulations does not play any role, we will consider \eqref{meanfield} and we will not comment on this issue any further.

\

From now on, we will consider potentials that do not change sign. We shall not impose any extra assumption on the sign of the parameters $\rho,\rho'$. For that reason, let us reduce ourselves to the case of positive potentials, namely (due to $\S$ and $\de\S$ being compact)

\beq\label{h}
\frac1C\le K(x)\le C\quad\quad\quad\forall x\in\S,\quad\quad\quad\quad\quad\quad\frac1C\le h(x)\le C\quad\quad\quad\forall x\in\de\S.\tag{H}
\eeq

Therefore, the sign of both right hand sides in \eqref{meanfield} will be determined by the sign of $\rho$ and $\rho'$ respectively.

Problem \eqref{meanfield} again admits a variational formulation, with the energy functional given by

\beq\label{j}
\boxed{\mathcal J(u)=\mathcal J_{\rho,\rho'}(u):=\int_\S\left(\frac12|\n u|^2+\frac{2(\rho+\rho')}{|\S|}u\right)-2\rho\log\int_\S Ke^u-4\rho'\log
\int_{\de\S}he^\frac u2.}
\eeq

\

A well-known tool to deduce crucial properties on the energy functional are the Moser--Trudinger type inequalities. Such inequalities show that the Sobolev emedding in the critical dimension is exponential, proved in the pioneer works of Moser and Trudinger \cite{tru,mos}. As a consequence, for \emph{closed} surfaces it holds:

\beq\label{mtineq}
16\pi\log\int_\S e^u\le\int_\S|\n u|^2+\frac{16\pi}{|\S|}\int_\S u+C,\quad\quad\quad\forall u\in H^1(\S).
\eeq

On the other hand, if $\S$ has a boundary, then the former inequality is no longer true and the constant $16\pi$ must be divided by $2$, as shown by Chang and Yang in \cite{chyg}:

\beq\label{chygineq}
8\pi\log\int_\S e^u\le\int_\S|\n u|^2+\frac{8\pi}{|\S|}\int_\S u+C,\quad\quad\quad\forall u\in H^1(\S).
\eeq

Finally, we recall a Moser--Trudinger type inequality with boundary integrals, rather than interior ones, given by Li and Liu in \cite{liliu}:

\beq\label{liliuineq}
16\pi\log\int_{\de\S}e^\frac u2\le\int_\S|\n u|^2+\frac{8\pi}{|\de\S|}\int_{\de\S}u+C,\quad\quad\quad\forall u\in H^1(\S).
\eeq

Interpolating \eqref{chygineq} and \eqref{liliuineq}, one derives the inequality

\beq\label{mt}
\boxed{4\rho\log\int_\S e^u+8(2\pi-\rho)\log\int_{\de\S}e^\frac u2\le\int_\S\left|\n u\right|^2+\frac{8\pi}{|\S|}\int_\S u+C_\rho,\quad\quad\quad\forall u\in H^1(\S).}
\eeq
for every $\rho\in [0,2\pi]$.

One of the aims of this paper is to extend the previous inequality to the case $\rho\le 4\pi$, including negative $\rho$, whose proof is not as immediate. Our approach is inspired by the ideas of \cite{jostwang,lucamt,lucandrea} to obtain other Moser--Trudinger type inequalities. In fact, such a proof is based on a blow--up analysis for sequences of minimizing solutions to \eqref{meanfield}.

\

By plugging \eqref{mt} in the definition of the energy functional \eqref{j} one easily obtains that $\mathcal J$ is coercive if $\rho<4\pi,\rho+\rho'<2\pi$. As an immediate consequence, in the coercivity case we can state the following existence result.

\begin{theorem}\label{ExiMin}
Assume \eqref{h}, $\rho<4\pi$ and $\rho+\rho'<2\pi$. Then, problem \eqref{meanfield} admits a solution which is a global minimizer of the energy functional $\mathcal J$ defined by \eqref{j}.
\end{theorem}

\

On the other hand, we also prove that inequality \eqref{mt} is somehow sharp, in the sense that it cannot hold either $\rho>4\pi$ or $2\pi-\rho$ is replaced by any larger number. Therefore, for any other choice of $\rho,\rho'$, critical points will not be global minimizers. However, we shall prove that there exist critical points of another type, such as saddle-type ones, by means of a min-max argument.

The strategy we will follow is based on the topological analysis of energetic sublevels

\beq\label{jl}
\mathcal J^L:=\left\{u\in H^1(\S):\mathcal J(u)\le L\right\}.
\eeq

This argument was first introduced in the pioneer paper \cite{dm} and it is now a rather classical tool in attacking mean field type problems, see \cite{bdm,andav,lucatoda,bjmr,jkm,lucand2,lucab2g2,dmls,dmlsr,lucaks}. The main ingredient is to prove that very low sublevels are non contractible. Actually, if $\mathcal J(u)\ll0$ then the measure $\frac{Ke^u}{\int_\S Ke^u}$ is almost concentrated at just a finite number of points, due to a localized version of the Moser--Trudinger inequality.

Therefore, low sublevels inherit some topology from such a space of finitely supported measures, called \emph{barycenters} which is not contractible under proper assumptions (a more detailed definition and description will be given later). On the other hand, an argument from \cite{luc} yields contractibility of very high sublevels, hence there must be some change of topology between them which implies, via Morse theory, existence of solutions.

This approach requires some compactness property for solutions. As a consequence of the blow--up analysis, one can deduce that a sequence of solutions to \eqref{meanfield} remains uniformly bounded if the parameter $\rho$ is not a multiple of $4\pi\N$ and $\rho+\rho'$ is not a multiple of $2\pi\N$.

Thus we choose $\rho$ between two consecutive multiples of $4\pi$ and $\rho+\rho'$ between two consecutive multiples of $2\pi$. Precisely, setting

\beq\label{nm}
N:=\inf\{n\in\N:\,\rho<4\pi(n+1)\},\quad\quad\quad\quad\quad\quad M:=\inf\{m\in\N:\,\rho+\rho'<2\pi(m+1)\};
\eeq
we will have

$$\left\{\begin{array}{ll}\rho<4\pi&\text{if }N=0\\
4\pi N<\rho<4\pi(N+1)&\text{if }N\in\N\end{array}\right.,\quad\quad\quad\quad\quad\quad\left\{\begin{array}{ll}\rho+\rho'<2\pi&\text{if }M=0\\
2\pi M<\rho<2\pi(M+1)&\text{if }M\in\N\end{array}\right..$$

\begin{theorem}\label{eximinmax}
Assume \eqref{h}, $\rho\notin4\pi\mathbb N$, $\rho+\rho'\notin2\pi\mathbb N$ and let $N,M\in\mathbb N\cup\{0\}$ be as in \eqref{nm}, with $(N,M)\ne(0,0)$.

If $\S$ is not simply connected, then problem \eqref{meanfield} admits solutions for any $N,M$.

If $\S$ is simply connected and $N<M$, then problem \eqref{meanfield} admits solutions.
\end{theorem}

Notice that the case $N=M=0$ is the one considered in Theorem \ref{mtteo}

\

Comparing with previous results obtained with similar methods, the interaction between $\S$ and its boundary plays a major issue. Roughly speaking, when $u$ concentrates at $\de\S$, then both exponential terms are affected, whereas when $u$ concentrates at the interior only one is affected.

For this reason, problem \eqref{meanfield} shares some similarities with systems of two equations having similar features on closed surfaces (see \cite{andav,bjmr,lucatoda,lucand2,lucab2g2,jkm}). However, in our case the interaction between interior and boundary nonlinearities is not symmetric, hence it needs to be treated differently than the case of systems.

In particular, the inequality given by Proposition \ref{prop0} will be essential in capturing the relation between interior and boundary concentration.

\

Another crucial novelty is given by the barycenters used to model concentration on both $\mathring\Sigma$ and $\de\Sigma$ (see \eqref{sigmajk}). Such objects seem to be rather mysterious and their topology has not been completely understood yet.

In order to overcome such an issue, we will use a new topological construction given in Proposition \ref{propretract}. Basically, such barycenters will be mapped on other spaces of barycenters, centered either only at points on $\mathring\Sigma$ or only at $\de\Sigma$. Since the homology of the latter barycenter spaces is well-known, we are finally able to deduce non-contractibility, hence existence of solutions.

\

\begin{remark}
In the paper \cite{lsmr} some obstructions are given to the existence of solutions to problem \eqref{gg} on multiply connected surfaces (Theorems 2.1 and 2.2), consequently also for \eqref{meanfieldgg}. However, Theorem \ref{eximinmax} gives existence of solutions for almost every $\rho$ in the non-geometric case.

This means that the third condition in \eqref{meanfieldgg} is not always satisfied by solutions to \eqref{meanfield} and that problems \eqref{meanfieldgg} and \eqref{meanfield} are indeed different.

However, it is reasonable to hope that the tools and results introduced here will be useful to solve related geometric problems, such as the prescription of Gaussian and geodesic curvatures in presence of conical singularities, see \cite{bdm,dmls,dmlsr} for more details.

\end{remark}

\

The content of the paper is the following: in Section 2 we give a suitable blow--up analysis for solutions and prove a concentration--compactness theorem; in Section 3 we prove a suitable Moser--Trudinger inequality; in Section 4 we show existence of min--max solutions.

\

We want to stress that references are cited in chronological order along this paper.

\

{\bf Notations.}

Let us fix some notations. The metric distance between two points $x,y\in\S$ will be denoted as $\dist(x,y)$. We will denote an open ball centered at a point $p\in\S$ of radius $r>0$ as 

$$B_r(p):=\{x\in\S:\,\dist(x,p)<r\}.$$

We will use the following notation for some subsets of $B_r(p)\subset\R^2$:
\begin{eqnarray*}
B_r^+(p)&:=&\left\{(x_1,x_2)\in B_r(p):\,x_2\ge0\right\};\\
\G_r(p)&:=&\left\{(x_1,x_2)\in\de B_r^+(p):\,x_2=0\right\};\\
\de^+B_r(p)&:=&\de B_r^+(p)\setminus\G_r(p).
\end{eqnarray*}
Given a subset $\Omega\subset\S$ and $\d>0$, we will denote the open $\d$-neighborhood of $\Omega$ as

$$\Omega^\d:=\{x\in\S:\,\dist(x,y)<\d,\,\text{for some}\,y\in\Omega\}.$$

Writing integrals, we will drop the element of area or length induced by the metric; for instance, we will only write $\int_\S Ke^u$ or $\int_{\de\S}he^\frac u2$.

In the estimates we will denote $C$ as a positive constant, independent of the parameters, that would vary from line to line. If we point out its dependence respect to certain parameters, we will indicate it in the subscript, such as $C_\e$ or $C_{\e,\d}$.

\

\section{A blow--up analysis}
\setcounter{equation}0

In this section we present some definitions and properties related to the blow--up of solutions to the problem \eqref{meanfield}. In some points we will be brief since some of the tools and results are rather classical and sometimes they only require minor changes.

Let us consider a sequence of solutions to

\beq\label{meanfieldn}
\left\{\begin{array}{ll}
-\Delta u_n+\frac{2(\rho_n+\rho'_n)}{|\S|}= 2\rho_n\frac{K_ne^{u_n}}{\int_\S K_ne^{u_n}}&\text{in }\S,\vspace{.3cm}\\
\de_\nu u_n=2\rho'_n\frac{h_ne^\frac{u_n}2}{\int_{\de\S}h_ne^\frac{u_n}2}&\text{on }\de\S.
\end{array}\right.
\eeq

The \emph{singular set} (see also \cite{breme}) is defined as

\beq\label{singularset}
\mathcal S:=\left\{p\in\S:\exists x_n\underset{n\to+\infty}\to p\text{ such that }u_n(x_n)\underset{n\to+\infty}\to+\infty\right\}.
\eeq

The pioneer works of \cite{breme,li-sha}, focused on the standard Liouville equation, show that around any isolated point $p\in\mathcal S$ one can rescale the solution and obtain in the limit an entire solution, which is classified. In this approach, the finite mass condition $\int_\S K_ne^{u_n}<+\infty$ plays a decisive role.

In fact, the classification of the solutions to the equation

$$-\Delta v=e^v\quad\quad\quad\text{ in }\R^2,$$
dates back to Liouville (\cite{liouville}), and they form a large family of nonexplicit solutions. However, under the finite mass assumption, all the solutions were classified in \cite{chen-li} and are given by

$$v(x)=\log\frac{8\l^2}{\left(1+\l^2|x-x_0|^2\right)^2},\quad\quad\quad\quad\quad\quad x_0\in\R^2,\quad\l>0.$$

Finite total mass implies also that the set $\mathcal S$ is finite, since $e^u$ behaves like a finite combination of Dirac deltas with weights bounded away from zero. The use of the Green's representation formula gives then some global information on the behavior of the solutions.

\

However, the presence of nonhomogeneous boundary conditions gives a second possible limit problem. As we will see, by a suitable rescaling, one can obtain the following limit problem

$$\left\{\begin{array}{ll}
-\Delta v=ae^v&\text{in }\R^2_+,\vspace{.3cm}\\
\de_\nu v=ce^\frac v2&\text{on }\de\R^2_+.
\end{array}\right.$$

Zhang classified the entire solutions on the half--plane in \cite{zhang} under the finite mass condition

$$\int_{\R^2_+}e^v<+\infty,\quad\quad\quad\quad\quad\quad\int_{\de\R^2_+}e^\frac v2<+\infty.$$

Actually, it is proved that if $a>0$, there exists solution for every $c$, whereas in case $a\le0$, then $c>\sqrt{-2a}$ and the explicit form is given by

$$v(x)=\log\frac{8\l^2}{\left(a+\l^2\left|x-\left(s_0,\frac c{\sqrt2\l}\right)\right|^2\right)^2},\quad\quad\quad\quad\quad\quad s_0\in\R,\quad\l>0$$

Although the shape of the solution depends on the sign of the constants $a$ and $c$, it is remarkable that for any $a,c$, they verify the quantization property

$$a\int_{\R^2_+}e^v+c\int_{\de\R^2_+}e^\frac v2=4\pi.$$

Let us emphasize that there exists a more general classification result, which includes unbounded mass solutions, given by Mira and G\'alvez (see \cite{mira-galvez} for further details). Notice that in the problem \eqref{meanfieldn} the mass is prescribed by the finite values $\rho_n$ and $\rho_n'$, therefore the infinite mass case will not be considered along this paper.

Next, we state the main result of this section.

\begin{theorem}\label{compa}
Let $u_n$ be a sequence of solutions to \eqref{meanfieldn} satisfying

\beq\label{massmass}
\frac1C\le\int_\S K_ne^{u_n}+\left(\int_{\de\S}h_ne^\frac{u_n}2\right)^2\le C,
\eeq
with $K_n\underset{n\to+\infty}\to K$, $h_n\underset{n\to+\infty}\to h$ in the $C^1$ sense such that $K,h$ verify \eqref{h} and $(\rho_n,\rho_n')\underset{n\to+\infty}\to(\rho,\rho')$. Define the singular set $\mathcal S$ as in \eqref{singularset}. Then, up to subsequences, the following alternative holds:

\begin{enumerate}
\item either $u_n$ is uniformly bounded in $L^\infty(\S)$;
\item or $\max_\S u_n\underset{n\to+\infty}\to+\infty$ and $\mathcal S$ is nonempty, finite and

\begin{eqnarray*}
2\rho_n\frac{K_ne^{u_n}}{\int_\S K_ne^{u_n}}&\underset{n\to+\infty}\weakto&8\pi\sum_{p\in\mathcal S\cap\mathring\S}\d_p+\sum_{q\in\mathcal S\cap\de\S}\b_q\d_q,\\
2\rho'_n\frac{h_ne^\frac{u_n}2}{\int_{\de\S}h_ne^\frac{u_n}2}&\underset{n\to+\infty}\weakto&\sum_{q\in\mathcal S\cap\de\S}(4\pi-\b_q)\d_q+\mu'_0,
\end{eqnarray*}
for some $\b_q\in\mathbb R$ and $\mu'_0\in L^1(\de\S)$, with $\mu'_0\equiv0$ if $\mathcal S\cap\de\S\ne\emptyset$.
\end{enumerate}
\end{theorem}

We point out that the previous theorem unifies the results given in \cite{bao,guoliu,lsmr}, which consider different assumptions on the sign of the potentials. Recall that in this formulation, the sign of the right hand side depends on $\rho,\rho'$ respectively.

Condition \eqref{massmass} is a \emph{normalization} which is needed to compensate the fact that problem \eqref{meanfieldn} is invariant by addition of constant. In fact, for any solution $u$ to \eqref{meanfieldn}, $u_n:=u\mp n$ still solves \eqref{meanfieldn} but clearly does not satisfy any of the alternative given by the previous theorem, as it goes to $+\infty$ or $-\infty$ everywhere on $\S$. However, given any solution $u_n$ one may add a suitable constant so that it satisfies $\frac{1}{C}\leq \int_\S K_ne^{u_n}+\left(\int_{\de\S}h_ne^\frac{u_n}2\right)^2\leq C$, therefore Theorem \ref{compa} is not restrictive in the study of solutions to \eqref{meanfieldn}

\begin{remark}
The very same result can be extended to the case when the sign of $h$ is constant on the connected components of $\de\S$, and similarly one can also extend Theorem \ref{eximinmax} to this case.

On the other hand, blow--up analysis seems more involved in the case of sign--changing potentials. A priori, compensation phenomena between masses may occur around zeroes of $K$ or $h$ (see for further details \cite{dmls,dmlsr}). Some results concerning blow--up analysis with sign--changing potentials have been provided in \cite{dmlsr}.
\end{remark}

\

From Theorem \ref{compa} one easily deduces that alternative \emph{(2)} can only occur if the parameters belong to the critical set, defined as (see Figure 1):

\beq\label{critic}
\boxed{\G=\left\{(\rho,\rho'):\rho=4\pi N\,\text{ or }\,\rho+\rho'=2\pi M\,\text{ with }N,M\in\mathbb N\right\}.}
\eeq

\begin{corollary}\label{notingamma}
If $(\rho,\rho')\notin\G$, then the set of solutions to the problem \eqref{meanfieldn} is compact up to addition of constants.
\end{corollary}

\begin{figure}[h!]
\centering
\includegraphics[width=.5\linewidth]{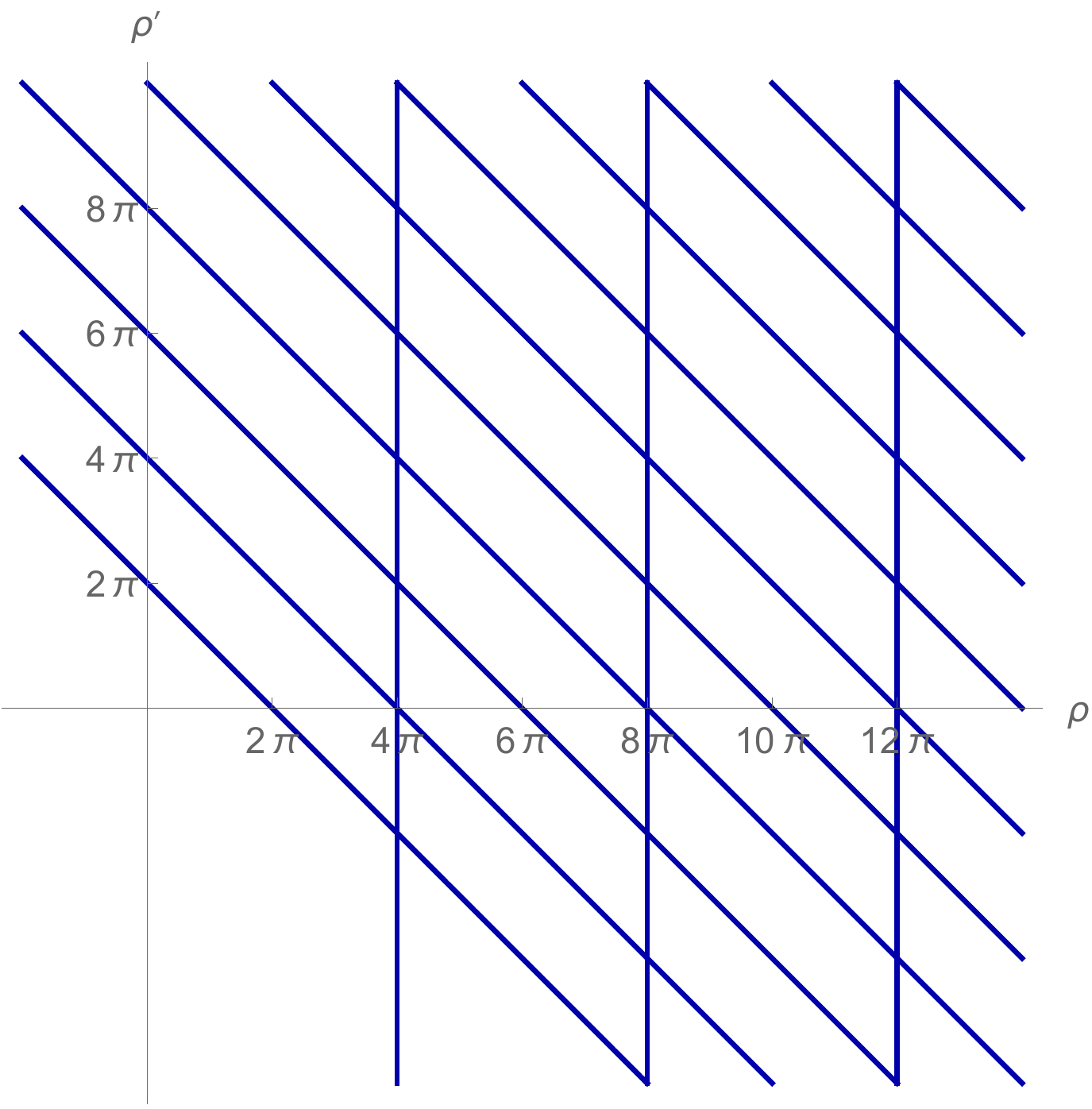}
\caption{The critical value set $\color[rgb]{0,0,.5}\G$}
\end{figure}

\

We will now collect some lemmas in order to prove Theorem \ref{compa}.

We will suppose that $u_n$ is a blowing--up sequence and $p$ is a point of $\mathcal S$. Moreover, by conformal invariance, we can map a neighborhood of $p$ to a disk $B_r(0)\subset\R^2$ if $p\in\mathring\S$ and $B_r^+(0)$ in case that $p\in\de\S$. Therefore, to study the local blow--up we suffice to consider the problem

\beq\label{limitproblem1}
{-\Delta u_n=2\widetilde K_ne^{u_n}}\quad\quad\quad\text{in }B_r(0),
\eeq
or

\beq\label{limitproblem2}
\left\{\begin{array}{ll}
-\Delta u_n=2\widetilde K_ne^{u_n}&\text{ in }B_r^+(0),\vspace{.3cm}\\
\de_\nu u_n=2\widetilde h_ne^\frac{u_n}2&\text{on $\G_r(0)$,}
\end{array}\right.
\eeq
for some $\widetilde K_n,\widetilde h_n$ having the same sign as $\rho_n,\rho_n'$ respectively.

\

A first step in the blow--up analysis is a minimal mass lemma, which implies finiteness of the blow--up set $\mathcal S$.

Such a result was first introduced by Brezis and Merle in \cite{breme} in the case of uniform Dirichlet conditions. We will need a proper version which takes into account the interior and the boundary integrals. Based on an adaptation for the boundary term given in \cite{jostwangz}, a first version was deduced in \cite{bao}. In the latter paper, only constant potentials $\widetilde K_n(x)\equiv\widetilde K_{n,0}>0$ are considered, which is directly adapted for non constant ones below. This argument does not deal with a possible compensation between masses. For that reason, it is an open question to find a sharp estimate.

\begin{lemma}\label{minmass}
Let $u_n$ be a sequence of solutions to \eqref{limitproblem1} for some $\widetilde K_n$ with $\widetilde K_n^+ \le C$. If

$$\int_{B_r^+(0)}\widetilde K^+_ne^{u_n}\le\e<2\pi,$$
then $u_n^+$ is uniformly bounded in $L^{\infty}\left(B_\frac r2(0)\right)$.

Let $u_n$ be a sequence of solutions to \eqref{limitproblem2} for some $\widetilde K_n,\widetilde h_n$ with $\widetilde K_n^+,\widetilde h_n^+\le C$. If

\beq\label{minmassb}
\int_{B_r^+(0)}\widetilde K_n^+e^{u_n} \le\e<\frac{\pi}{2}, \quad  \int_{\G_r(0)}\widetilde h_n^+e^\frac{u_n}2\le\e<\frac{\pi}{2},
\eeq
then $u_n^+$ is uniformly bounded in $L^{\infty}\left(B_\frac r2^+(0)\right)$.

\end{lemma}

\begin{proof}

The first statement follows directly from \cite{breme}, Corollary 4.

For the second one, let us decompose $u_n=u_{1n,+}+u_{1n,-}+u_{2n,+}+u_{2n,-}+u_{3n}$, where $u_{1n,\pm}$ and $u_{2n,\pm}$ satisfy

$$\left\{\begin{array}{ll}
-\Delta u_{1n, \pm}=2\widetilde K_n^{\pm}e^{u_n}=\widetilde f_n^\pm &\text{ in }B_r^+(0),\vspace{.3cm}\\
\de_\nu u_{1n,\pm}=0&\text{on }\G_r(0),\vspace{.3cm}\\
u_{1n,\pm}=0&\text{on }\de^+B_r(0),
\end{array}\right.$$
and

$$\left\{\begin{array}{ll}
\Delta u_{2n,\pm}=0&\text{ in }B_r^+(0),\vspace{.3cm}\\
\de_\nu u_{2n,\pm}=2\widetilde h_n^\pm e^\frac{u_n}2=\widetilde g_n^\pm &\text{on }\G_r(0),\vspace{.3cm}\\
u_{2n}=0&\text{on }\de^+ B_r(0).
\end{array}\right.$$

Clearly, $u_{1n,-},u_{2n,-}\leq0$. Extending $u_{1n,+}$ and $\widetilde f_n^+$ evenly in $B_r(0)$, we can use Corollary 4 of \cite{breme} in order to obtain that

$$\int_{B_r^+(0)}e^{p|u_{1n,+}|}\le C,$$
for some $p>1$. On the other hand, in view of \cite{jostwangz}, Lemma 3.2, we have that, for $0<\d_1<4\pi$ and $0<\d_2<2\pi$ then

$$\int_{B_r^+(0)}e^{\frac{4\pi-\d_1}{\left\|\widetilde g_n^+\right\|_{L^1}}|u_{2n,+}|}\le C\quad\quad\quad\text{and}\quad\quad\quad\int_{\G_r(0)}e^{\frac{2\pi-\d_2}{\left\|\widetilde g_n^+\right\|_{L^1}}|u_{2n,+}|}\le C,$$
where $\left\|\widetilde g_n^+\right\|_{L^1}=\int_{\G_r(0)}\widetilde g_n^+$.

Since $\left\|\widetilde g_n^+\right\|_{L^1}<2\pi$, then

$$\int_{B_r^+(0)}e^{p_1|u_{2n,+}|}\le C\quad\quad\quad\text{and}\quad\quad\quad\int_{\G_r(0)}e^{p_2|u_{2n,+}|}\le C,$$
for some $p_1,p_2>1$. Consequently, $u_{1n,+},u_{2n,+}\in L^1\left(B_r^+(0)\right)$ and, moreover, $u_n^+\in L^1\left(B_r^+(0)\right)$ by the boundedness mass assumption. Notice that $u_{3n}^+= u_n^+ +u_{1n,+}+u_{2n,+}$, so applying the mean value theorem, one gets $u^+_{3,n}\in L^{\infty}\left(B_\frac r2^+(0)\right)$.

Therefore, we have obtained that $\widetilde f_n\in L^q\left(B_\frac r2^+(0)\right)$ and $\widetilde g_n\in L^q\left(\G_\frac r2(0)\right)$ for $q>1$. Finally, standard elliptic estimates provide the conclusion. 

\end{proof}

\

By Lemma \ref{minmass}, we conclude that if $u_n$ blows up somewhere in $\S$, then the \emph{mass} is at least $4\pi$ at the interior and there exist boundary contributions to be precised. Therefore the blow--up set must be finite, as we are assuming the masses to be finite. Moreover, we get

\begin{eqnarray}\label{wconv}
2\rho_n\frac{K_ne^{u_n}}{\int_\S K_ne^{u_n}}&\underset{n\to+\infty}\weakto&\sum_{p\in\mathcal S\cap\mathring\S}\a_p\d_p+\sum_{q\in\mathcal S\cap\de\S}\b_q\d_q+\mu_0,\vspace{.3cm}\\
\nonumber2\rho'_n\frac{h_ne^\frac{u_n}2}{\int_{\de\S}h_ne^\frac{u_n}2}&\underset{n\to+\infty}\weakto&\sum_{q\in\mathcal S\cap\de\S}\g_q\d_q+\mu'_0,
\end{eqnarray}
for some $\a_p,\b_q,\g_q$ with $\a_p\ge4\pi,\b_q,\g_q\in\mathbb{R}$ and $\mu_0\in L^1(\S),\mu'_0\in L^1(\de\S)$.

\

\

Next, we have a quantization of local blow--up masses. Whereas internal mass at each concentration point corresponds to $4\pi$, in case of blow--up at the boundary only the sum of the two masses is quantized.

The proof is rather similar as \cite{bao} (Theorem 1.5), \cite{lsmr} (Lemma 7.5), therefore it will be skipped. It is based on a Poho\v zaev identity around the blow--up point, the inequality

$$ u_n+2\log r_n \to - \infty \quad \mbox{ if $\dist(x,p)= r_n$}$$
for a proper $r_n\to0$ and $p\in \mathcal{S}$ (see Lemma~3.1. in \cite{lwz} for more details) and a uniformly mean oscillation property, namely

$$\max_{\de B_r(p)\cap\S}u_n-\min_{\de B_r(p)\cap\S}u_n\le C.$$

\begin{lemma}\label{quant}
Define, for $p\in\mathcal S$, its \emph{blow--up value} as

$$m(p):=\lim_{r\to0}\lim_{n\to+\infty}\left(\rho_n\frac{\int_{B_r(p)}K_ne^{u_n}}{\int_\S K_ne^{u_n}}+\rho'_n\frac{\int_{B_r(p)\cap\de\S}h_ne^\frac{u_n}2}{\int_{\de\S}h_ne^\frac{u_n}2}\right).$$

Then, $m(p)=4\pi$ if $p\in\mathring\S$ and $m(p)=2\pi$ if $p\in\de\S$.
\end{lemma}

\begin{remark}
If $\rho\le0$, then there is no interior blow--up, i.e. $\mathcal S\subset\de\S$. This is a direct consequence of the first statement of Lemma \ref{minmass}. It is possible to deduce this property by means of the maximum principle.
\end{remark}

\

Finally, we focus on the residuals $\mu_0$ and $\mu'_0$.

In the classical case without boundary nonlinear terms, studied by \cite{breme}, there is no residual in case of blow--up. Here, the situation is similar concerning the internal mass, but one may still have a residual mass on the boundary because of the fixed amount of mass.

\begin{proposition}\label{residual}
Assume that the singular set $\mathcal S$ is not empty. Then $\mu_0\equiv0$ in \eqref{wconv}.

Moreover, if $\mathcal S\cap\de\S\ne\emptyset$, then $\mu'_0\equiv0$ in \eqref{wconv}.
\end{proposition}

\begin{proof}
First we consider the interior blow--up case, namely $p\in\mathcal S\setminus\de\S$. Without loss of generality suppose that $u_n$ is solution to \eqref{limitproblem1} and $p=\{0\}$. Next, take $r$ small enough such that $\overline{B_r(0)}\cap\mathcal S=\{0\}$.

By contradiction, assume that $u_n\ge-C$ on $\de B_r(0)$. Since $u_n$ is uniformly bounded in $C^{2,\g}(B_r(0)\setminus\{0\})$, it is possible to pass to the limit, up to subsequences, and to obtain

$$\mu_n=2\widetilde K_ne^{u_n}\underset{n\to+\infty}\weakto\mu$$
in the sense of measures on $B_r(0)$ and

$$u_n\underset{n\to+\infty}\to\eta\text{ in }C^2_\mathrm{loc}(B_r(0)\setminus\{0\}),$$
where $\eta$ is a weak solution to $-\Delta\eta=\mu$. By Lemma \ref{quant}, then $\mu(\{0\})\ge8\pi$, hence by the Green's representation formula, we finally arrive at

\beq\label{contr}
\eta \ge-4\log|x|-C\quad\text{in }B_r(0),
\eeq
which is in contradiction with $\int_{B_r(0)}\widetilde K_ne^{u_n}\le C$. Therefore, $u_n\underset{n\to+\infty}\to-\infty$ uniformly in compacts sets of $B_r(0)\setminus\{0\}$ and in particular $\mu_0\equiv0$.

If $p\in\mathcal S\cap\de\S$, we can repeat the previous argument. By a conformal transformation, we can pass to the problem \eqref{limitproblem2} and then apply Lemma~\ref{quant}. Arguing by contradiction, as before $u_n\underset{n\to+\infty}\to\eta$ in, and there exists $z\le \eta$ which satisfies the problem

\begin{equation} \label{z}
\left\{\begin{array}{ll}
-\Delta z=4\pi\d_q+\mu_0\quad&\text{in }B_r^+(0),\\
\de_\nu z=\mu_0'\quad&\text{on }\G_r(0),\\
z=-C\quad&\text{on }\de^+B_r(0).
\end{array}\right.
\end{equation}

Now, let us distinguish between cases depending on the signs of $\mu_0, \mu_0'$.

If $\mu_0,\mu_0'\ge0$, then by Green's representation formula gives \eqref{contr} which violates the condition $\int_{\G_r(0)}\widetilde h_ne^\frac{u_n}2\le C$. Therefore, we deduce that $u_n\underset{n\to+\infty}\to-\infty$ locally uniformly in $B_r^+(0)\setminus\{0\}$ and so $\mu_0=\mu'_0\equiv0$.
 
Suppose now that $\mu_0\le0$.We write $z=z_1+z_2$, where

\begin{equation}\label{z1}
\left\{\begin{array}{ll}
\Delta z_1=0\quad&\text{in }B_r^+(0),\\
\de_\nu z_1=4\pi\d_q+\mu_0'\quad&\text{on }\Gamma_r(0),\\
z_1=-C\quad&\text{on }\partial^+B_r(0),
\end{array}\right.
\end{equation}
and 

\begin{equation}\label{z2}
\left\{\begin{array}{ll}
-\Delta z_2=\mu_0\quad&\text{in }B_r^+(0),\\
\de_\nu z_2=0\quad&\text{on }\Gamma_r(0),\\
z_2=0\quad&\text{on }\partial^+B_r(0).
\end{array}\right.
\end{equation} 

On one hand, by using the Green's representation formula, we obtain that 

$$z_1\ge-4\log|x|-C.$$

On the other hand, applying the minimal mass lemma by Brezis-Merle, namely Theorem~1 in \cite{breme}, one has that

$$\int_{B_r^+(0)}e^{p|z_2|}\le C,$$
for some $p>1$, which can be enlarged if it is necessary.

Therefore, we arrive at

$$\int_{B_r^+(0)}\frac C{|x|^{\frac4q}}\le\int_{B_r^+(0)}e^{\frac{z_1}q}=\int_{B_r^+(0)}e^{\frac{z-z_2}q}\le\left(\int_{B_r^+(0)}e^z\right)^{\frac1q}\left(\int_{B_r^+(0)}e^{\frac{|z_2|}{q-1}}\right)^\frac{q-1}q\le C,$$
by H\"{o}lder inequality with $1<q<2$. Taking $\frac{1}{q-1}<p$, the previous inequality gives us a contradiction. Therefore, $u_n\to -\infty$ and $\mu_0=0=\mu_0'$

\

Similarly, we obtain the same conclusion if $\mu'_0\le0$

\end{proof}

\

Now we are in position to prove Theorem \ref{compa}

\begin{proof}[Proof of Theorem \ref{compa}] Let us differentiate two cases:

\

\textbf{Case 1:} The sequence $u_n$ is bounded from above in $\S$.

Then $K_ne^{u_n}\in L^p(\S)$ and $h_ne^\frac{u_n}2\in L^p(\de\S)$ for some $p>1$, therefore $-\Delta u_n\in L^p(\S)$ and $\de_\nu u_n\in L^p(\de\S)$. Using elliptic regularity estimates, one can deduce that $u_n-\overline u_n\in W^{2,p}$, and in particular $u_n-\overline u_n\in L^{\infty}(\S)$.
If $\overline u_n$ remains bounded, then we obtain \emph{(1)}. Otherwise, a sequence of $u_n$ diverges negatively in $\S$ which violates the condition \eqref{massmass}.

\

\textbf{Case 2:} $u_n$ is not bounded from above.

Applying Lemma \ref{minmass}, we know that $\mathcal S$ is nonempty and finite. Moreover, combining Proposition \ref{residual} and Lemma \ref{quant} we finally deduce \emph{(2)}.

\end{proof}

\

An important consequence of the blow--up analysis is the following result, which will be essential in the forthcoming sections.

\begin{proposition}\label{quasisharp}
Let $u_n$ be a sequence of solutions to \eqref{meanfieldn} under the assumptions of Theorem \ref{compa} with $\mathcal S\cap\de\S\neq\emptyset$, $\rho=2\pi-\rho'<4\pi$ and $\rho\ne0,2\pi$. Then,

$$\frac1C\le\frac{\sqrt{\int_\S K_ne^{u_n}}}{\int_{\de\S}h_ne^\frac{u_n}2}\le C,$$
for some positive constant $C$.
\end{proposition}

\begin{proof}
By the choice of $\rho,\rho'$, as a consequence of Theorem \ref{compa} we have that $\mathcal S=\{p\}\subset\de\S$. Next, define

$$v_n:=u_n-2\log\min\left\{\sqrt{\int_\S K_ne^{u_n}},\int_{\de\S}h_ne^\frac{u_n}2\right\},$$
which satisfies the equation

$$\left\{\begin{array}{ll}
-\Delta v_n+\frac{2(\rho_n+\rho'_n)}{|\S|}=2\rho_n\widetilde K_ne^{v_n}&\text{in }\S,\vspace{.3cm}\\
\de_\nu v_n=2\rho'_n\widetilde h_ne^\frac{v_n}2&\text{on }\de\S.
\end{array}\right.$$
where

$$\begin{array}{rlll}
\text{either}&\quad\quad\widetilde K_n=K_n&\quad\quad\text{and}&\quad\quad\widetilde h_n=h_n\frac{\sqrt{\int_\S K_ne^{u_n}}}{\int_{\de\S}h_ne^\frac{u_n}2},\\
\text{or}&\quad\quad\widetilde K_n=K_n\frac{\left(\int_{\de\S}h_ne^\frac{u_n}2\right)^2}{\int_\S K_ne^{u_n}}&\quad\quad\text{and}&\quad\quad\widetilde h_n=h_n.
\end{array}$$

Clearly,

\beq\label{mass1}
\int_\S\widetilde K_ne^{v_n}=1\quad\text{and}\quad\int_{\de\S}\widetilde h_ne^\frac{v_n}2=1.
\eeq

Notice that $0<\widetilde K_n,\widetilde h_n\le C$ and $v_n$ blows up at the point $p$ and $\min\left\{\sqrt{\int_\S K_ne^{u_n}},\int_{\de\S}h_ne^\frac{u_n}2\right\}$ is uniformly bounded due to assumption \eqref{massmass}. Now, we take 

$$v_n(x_n)=\max_\S v_n\underset{n\to+\infty}\to+\infty\quad\quad\quad\text{and}\quad\quad\quad\d_n:=e^\frac{-v_n(x_n)}2\underset{n\to+\infty}\to0,$$
hence $x_n\underset{n\to+\infty}\to p$. Next, consider the following rescaling

$$\widetilde v_n=v_n(\d_n x+x_n)+2\log\d_n,$$
which verifies the equation 

\beq\label{problemquant}
\left\{\begin{array}{ll}
-\Delta\widetilde v_n+\frac{2(\rho_n+\rho'_n)}{|\S|}\d_n^2= 2\rho_n\widetilde K_n(\d_n x+x_n)e^{\widetilde v_n}&\text{in $B_\frac{r_n}{\d_n}^+(0)$,}\vspace{.3cm}\\
\de_\nu\widetilde v_n=2\rho'_n\widetilde h_n(\d_n x+x_n)e^{\frac{\widetilde v_n}2}&\text{on $\G_\frac {r_n}{\d_n} (0)$,}
\end{array}\right.
\eeq
where $B_\frac{r_n}{\d_n}^+(0)=B_\frac{r_n}{\d_n}(0)\cap\S$ for some $r_n>0$ to be fixed.

\

Now, for $r>0$, let us distinguish two cases:

\

\textbf{Case 1}: $\frac{\dist\left(x_n,\G_\frac r{\d_n}(0)\right)}{\d_n}\underset{n\to+\infty}\to+\infty.$ 

In this case, let us set $r_n=\frac{|x_n|}2$ in \eqref{problemquant}, so that $\frac{r_n}{\d_n}\to+\infty$. Moreover, $\widetilde v_n\le0$ and it is uniformly bounded in $L^\infty_\mathrm{loc}$ by Harnack inequality. Therefore, it is standard to prove that

$$\widetilde v_n\underset{n\to+\infty}\to\widetilde v\quad\text{in }C^2_\mathrm{loc}\left(\R^2\right),$$
with $\widetilde v$ solving

$$\left\{\begin{array}{ll}
-\Delta\widetilde v=2\rho\widetilde K(p)e^{\widetilde v}\quad&\text{in }\R^2,\vspace{.3cm}\\
\int_{\R^2}\widetilde K(p)e^{\widetilde v}\le C\end{array}\right.$$

However, by \eqref{mass1} and the classification of the solutions to the previous problem, we conclude that

$$\rho_n=\int_\S\rho_n\widetilde Ke^{v_n}\ge\int_{B_\frac{|x_n|}{2\d_n}(x_n)}\rho_n\widetilde Ke^{v_n}\ge\int_{\R^2}\rho\widetilde K(p)e^{\widetilde v}=4\pi,$$
which is a contradiction with the choice of $\rho$.

\

\textbf{Case 2:} $\dist\left(x_n,\G_\frac r{\d_n}(0)\right)=O(\d_n)$.

Now, we consider a sequence $\widetilde v_n$ of solutions to \eqref{problemquant} with $r_n=r>0$. Again, $\widetilde v_n\le0$ and it is uniformly bounded in $L^\infty_\mathrm{loc}$ by Harnack type inequality for \eqref{problemquant} (see \cite{jostwzz}, Lemma A.2). Consequently, up to a subsequence and after a proper translation, we get that

$$\widetilde v_n\underset{n\to+\infty}\to\widetilde v\quad\text{in } C^2_\mathrm{loc}\left(\R^2_+\right),$$
with $\widetilde v$ solving

$$\left\{\begin{array}{ll}
-\Delta\widetilde v=2\rho\widetilde K(p)e^{\widetilde v}&\text{in }\R^2_+,\vspace{.3cm}\\
\de_\nu\widetilde v=2\rho'\widetilde h(p)e^\frac{\widetilde v}2&\text{on }\de\R^2_+,\vspace{.3cm}\\
\int_{\R^2_+}e^{\widetilde v}\le C,\quad\quad\quad\int_{\de\R^2_+}e^\frac{\widetilde v}2\le C
\end{array}\right.$$
satisfying the quantization property $\int_{\R^2_+}\rho\widetilde K(p)e^{\widetilde v}+\int_{\de\R^2_+}\rho'\widetilde h(p)e^\frac{\widetilde v}2=2\pi$.

$$\frac1C\le\widetilde K(p),\widetilde h(p)\le C.$$

Taking into account the definition of $\widetilde K$ and $\widetilde h$ and \eqref{h}, the last fact implies directly the desired conclusion.
\end{proof}

\begin{remark}\label{remark}
This result can be extended for every critical value $(\rho,\rho')\in\G$ such that $\rho\notin4\pi\N$ and $\mathcal S\cap\de\S\neq\emptyset$. On the other hand, if $\rho\in4\pi\N$, then the blow--up may approach the boundary so slowly that the ratio between the masses is not controlled.
\end{remark}

\

\section{A Moser--Trudinger type Inequality}
\setcounter{equation}0

We will now start to study variationally problem \eqref{meanfield}. In particular, this section is devoted to prove a Moser--Trudinger inequality \eqref{mt}, which generalizes previous inequalities \eqref{chygineq}, \eqref{liliuineq}.

Precisely, we have the following theorem:

\begin{theorem}\label{mtteo} Inequality $\eqref{mt}$ holds true. In particular, the energy functional $\mathcal J_{\rho,\rho'}$ is:
\begin{enumerate}
\item bounded from below on $H^1(\S)$ if $\rho\le4\pi$, $\rho+\rho'\le2\pi$ and $(\rho,\rho')\ne(4\pi,-2\pi)$;
\item coercive on $\overline H^1(\S)$ if and only if $\rho<4\pi,\rho+\rho'<2\pi$.
\end{enumerate}
\end{theorem}
Here, $\overline H^1(\S)$ is the subspace of $H^1(\S)$ containing only functions with zero average

$$\overline H^1(\S):=\left\{u\in H^1(\S):\,\int_\S u=0\right\}.$$

Notice that the energy functional $\mathcal J_{\rho,\rho'}$ can never be coercive on the whole space $H^1(\S)$, since it is invariant by addition of constant, as well as equation \eqref{meanfield} is. On the other hand, one can choose $\sqrt{\int_\S|\n u|^2}$ as an equivalent norm on $\overline H^1(\S)$, due to Poincar\'e inequality. For this reason, we will be discussing coercivity of $\mathcal J$ on $\overline H^1(\S)$ and we will omit the space we are considering.

\

As an immediate corollary of Theorem~\ref{mtteo}, in the coercivity case we get minimizing solutions to the double mean field equation, that is Theorem~\ref{ExiMin} is proved.

\

By comparing Theorem \ref{mtteo} with the blow--up analysis from Section 2, we see that $\mathcal J_{\rho,\rho'}$ is coercive for $\rho,\rho'$ in an open region disjointed from the critical value set defined by \eqref{critic}. This is not a coincidence, as Corollary \ref{notingamma} will be used in the proof of Theorem \ref{mtteo}.

\begin{figure}[h!]
\centering
\includegraphics[width=.5\linewidth]{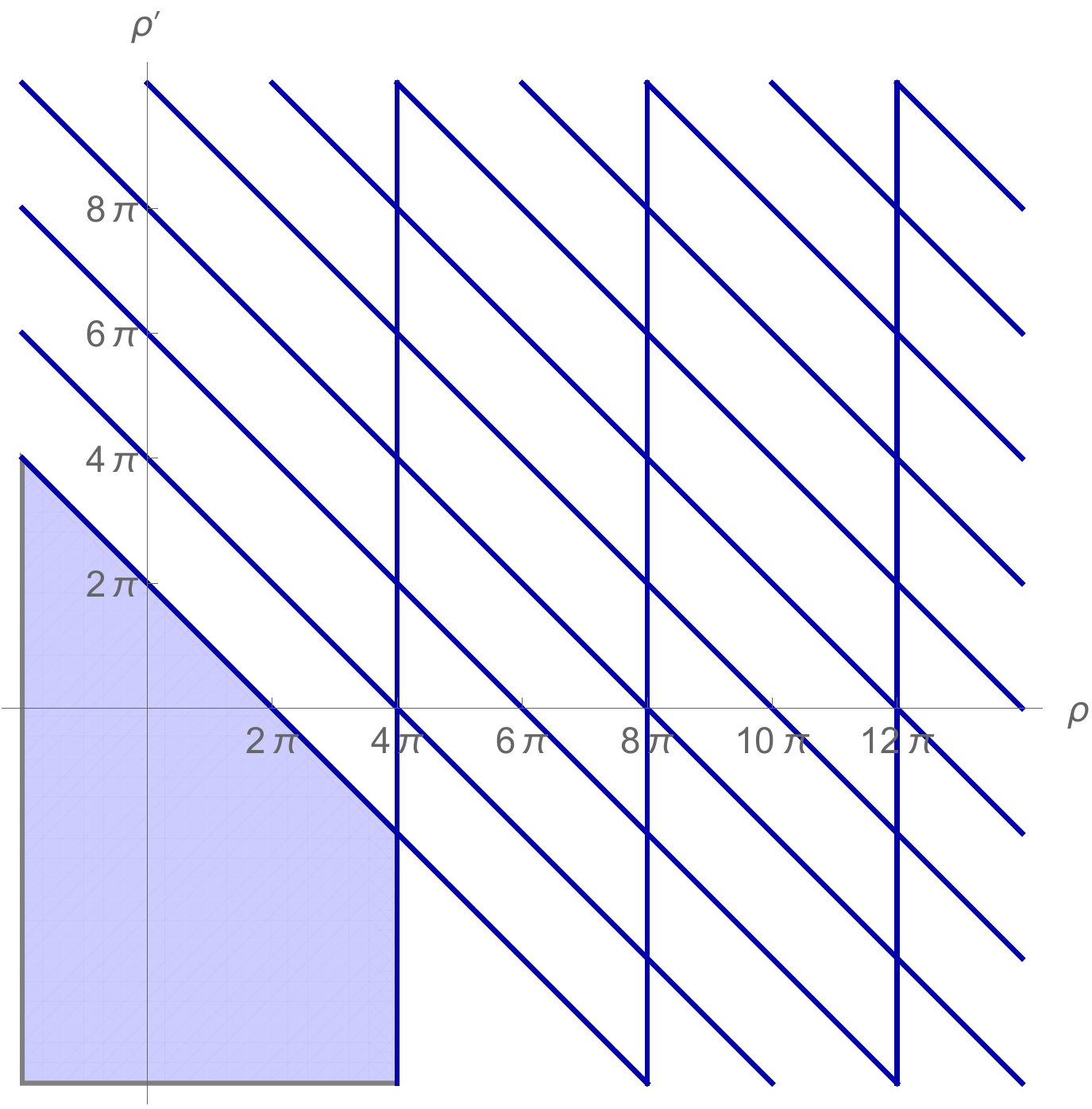}
\caption{The range of parameters for which $\mathcal J_{\rho,\rho'}$ is coercive and the critical set.}
\end{figure}

\

The \emph{only if} part in Theorem \ref{mtteo} is rather easy to be proved, using suitable test functions.

\begin{proposition}\label{nonmt}
If either $\rho>4\pi$ or $\rho+\rho'>2\pi$, then $\inf_{u\in H^1(\S)}\mathcal J_{\rho,\rho'}=-\infty$.

If either $\rho\ge4\pi$ or $\rho+\rho'\ge2\pi$, then $\mathcal J_{\rho,\rho'}$ is NOT coercive.
\end{proposition}

\begin{proof}
It is enough to consider, for $\l>0$ and $p\in\S$,

$$\phi^\l_p(x)=\log\left(\frac{\l}{1+\l^2\dist^2(x,p)}\right)^2.$$

One easily gets the following and well--known estimates (see for instance \cite{dm,guoliu}):

\begin{eqnarray*}
\int_\S\left|\n\phi^\l_p\right|^2&=&\left\{\begin{array}{ll}32\pi\log\l+O(1)&\text{if }p\in\mathring\S,\vspace{.3cm}\\16\pi\log\l+O(1)&\text{if }p\in\de\S;\end{array}\right.\\
\log\int_\S Ke^{\phi^\l_p}&=&O(1);\\
\log\int_{\de\S}he^{\frac{\phi^\l_p}2}&=&\left\{\begin{array}{ll}-\log\l+O(1)&\text{if }p\in\mathring\S,\vspace{.3cm}\\O(1)&\text{if }p\in\de\S;\end{array}\right.\\
\int_\S\phi^\l_p&=&-2|\S|\log\l+O(1).
\end{eqnarray*}

Therefore, by taking $p\in\mathring\S$ and evaluating $\mathcal J_{\rho,\rho'}$ one gets

$$\mathcal J_{\rho,\rho'}\left(\phi^\l_p\right)\underset{\l\to+\infty}\to-\infty\quad\text{if }\rho>4\pi;\quad\quad\quad\quad\quad\quad\mathcal J_{\rho,\rho'}\left(\phi^\l_p\right)\le C<+\infty\quad\text{if }\rho\ge4\pi.$$

On the other hand, if $p\in\de\S$, the same calculations give:

$$\mathcal J_{\rho,\rho'}\left(\phi^\l_p\right)\underset{\l\to+\infty}\to-\infty\quad\text{if }\rho+\rho'>2\pi;\quad\quad\quad\quad\quad\quad\mathcal J_{\rho,\rho'}\left(\phi^\l_p\right)\le C<+\infty\quad\text{if }\rho+\rho'\ge2\pi;$$
this concludes the proof.
\end{proof}

\

As a first step, by mixing inequalities \eqref{chygineq} and \eqref{liliuineq}, one easily obtains Theorem \ref{mtteo} for a small range of parameters.

\begin{lemma}\label{interpol}
The energy functional $\mathcal J_{\rho,\rho'}$ is bounded from below if $\rho,\rho'\le2\pi$ and $\rho+\rho'\le2\pi$.
\end{lemma}

\begin{proof}
Multiplying inequality \eqref{chygineq} by $\frac\rho{2\pi}$, inequality \eqref{mtineq} by $\frac{\rho'}{2\pi}$ and adding both underlying expressions, one has that

$$4\rho\log\int_\S e^u+8\rho'\log\int_{\de\S}e^\frac u2\le\frac{\rho+\rho'}{2\pi}\int_\S\left|\n u\right|^2+\frac{4\rho}{|\S|}\int_\S u+\frac{4\rho'}{|\de\S|}\int_{\de\S}u+C_\rho.$$

Now, we can estimate the mean value on $\de\S$ in terms of the mean value on $\S$, as in \cite{sergio}. Precisely, we apply the previous inequality to $u-\vfi$, with $\vfi\in C^2(\S)$ solving \eqref{phi}. Therefore,

\begin{eqnarray*}
4\rho\log\int_\S Ke^u+8\rho'\log\int_{\de\S}he^\frac u2&\le&4\rho\log\int_\S e^{u-\vfi}+8\rho'\log\int_{\de\S}e^\frac{u-\vfi}2+C\\
&\le&\frac{\rho+\rho'}{2\pi}\int_\S\left|\n(u-\vfi)\right|^2+\frac{4\rho}{|\S|}\int_\S(u-\vfi)+\frac{4\rho'}{|\de\S|}\int_{\de\S}(u-\vfi)+C_\rho.\\
&\le&\int_\S\left(|\n u|^2+|\n\vfi|^2-2\n u\cdot\n\vfi\right)+\frac{4\rho}{|\S|}\int_\S u+\frac{4\rho'}{|\de\S|}\int_{\de\S}u+C\\
&\le&\int_\S|\n u|^2+2\left(\int_\S u\Delta\vfi-\int_{\de\S}u\de_\nu\vfi\right)+\frac{4\rho}{|\S|}\int_\S u+\frac{4\rho'}{|\de\S|}\int_{\de\S}u+C\\
&=&\int_\S|\n u|^2+\frac{4(\rho+\rho')}{|\S|}\int_\S u+C;
\end{eqnarray*}
namely, $\mathcal J_{\rho,\rho'}\ge-\frac{C}2$.

\end{proof}

\

We have a characterization of the values for which $\mathcal J$ is \emph{super--quadratic}, hence coercive. See also \cite{jostwang} (Lemmas 4.2, 4.3), \cite{lucandrea} (Lemmas 3.4, 3.5), \cite{lucamt} (Lemmas 4.3, 4.4).

\begin{lemma}\label{aperto}
Define $\L$ as the set of parameters for which $\mathcal J$ is bounded from below, namely

\beq\label{lambda}
\L:=\left\{(\rho,\rho')\in\R^2:\,\mathcal J_{\rho,\rho'}\text{ is bounded from below}\right\}.
\eeq

Then, $(\rho,\rho')$ belongs to the interior $\mathring\L$ of $\L$ if and only if there exists $C>0$ such that

\beq\label{coerc}
\mathcal J_{\rho,\rho'}(u)\ge\frac1C\int_\S|\n u|^2-C\quad\quad\quad\forall u\in H^1(\S).
\eeq

In particular, the set of parameters for which $\mathcal J$ is coercive is open and coincides with $\mathring\L$.
\end{lemma}

\begin{proof}
First of all, due to Jensen inequality one has

\beq\label{jensen}
\log\int_\S e^u\ge\frac1{|\S|}\int_\S u+\log|\S|\quad\quad\quad\quad\quad\quad\log\int_{\de\S}e^\frac u2\ge\frac1{|\de\S|}\int_{\de\S}u+\log|\de\S|.
\eeq

Therefore, one sees that if $(\rho,\rho')\in \L$ then $\left(\widetilde\rho,\rho'\right)\in\L$ for any $\widetilde\rho\le\rho$. Moreover, one can argue as in the proof of Lemma \ref{interpol} and compute $\mathcal J_{\rho,\rho'}$ on $u$ plus a suitable multiple of the solution $\vfi$ to \eqref{phi}, in order to switch the mean value on $\S$ and the mean value on $\de\S$; using this and the second inequality in \eqref{jensen}, one gets that $\left(\widetilde\rho,\widetilde\rho'\right)\in\L$ if $\widetilde\rho'\le\rho'$.

\

We are left with showing that $(1+\d)(\rho,\rho')\in\L$ if and only if \eqref{coerc} holds.

This will follow by writing

$$\mathcal J_{(1+\d)\rho,(1+\d)\rho'}(u)=(1+\d)\mathcal J_{\rho,\rho'}(u)-\frac\d2\int_\S|\n u|^2;$$
in fact, this implies that $\mathcal J_{(1+\d),(1+\d)\rho'}\ge-C$ if and only if $\mathcal J_{\rho,\rho'}(u)\ge\frac{\d}{2(1+\d)}\int_\S|\n u|^2$.

\end{proof}

\

Following the approach of \cite{jostwang}, we introduce an auxiliary perturbation of $\mathcal J_{\rho,\rho'}$ in the case of non-coercive parameters. This is particularly useful because the results concerning the blow--up can be applied to its critical points.

\begin{lemma}\label{perturb}
Assume that for $(\rho_0,\rho'_0)$ there exists a sequence $u_n\in H^1(\S)$ such that

\beq\label{noncoerc}
\int_\S|\n u_n|^2\underset{n\to+\infty}\to+\infty,\quad\quad\quad\quad\quad\quad\liminf_{n\to+\infty}\frac{\mathcal J_{\rho_0,\rho'_0}(u_n)}{\int_\S|\n u_n|^2}\leq 0.
\eeq

Then, there exists a smooth $F:H^1(\S)\to\R$ and such that, up to subsequences,

\beq\label{f}
0<F'(u)<1,\quad\quad\quad F'(u)\underset{\int_\S|\n u|^2\to+\infty}\to0\quad\text{uniformly}\quad\quad\quad\inf_{H^1(\S)}\left(\mathcal J_{\rho_0,\rho'_0}-F\right)=-\infty.
\eeq

\end{lemma}

\begin{proof}
An elementary calculus lemma (see for instance \cite{jostwang}, Lemma 4.4) states the following: for any two real sequences $\{a_n\},\{b_n\}$ satisfying

$$a_n\underset{n\to+\infty}\to+\infty,\quad\quad\quad\quad\quad\quad\liminf_{n\to+\infty}\frac{b_n}{a_n}\le0,$$
there exists a smooth $f:[0,+\infty)\to\R$ such that

$$0<f'(t)<1,\quad\quad\quad\quad\quad\quad f'(t)\underset{t\to+\infty}\to0,\quad\quad\quad\quad\quad\quad f(a_n)-b_n\underset{n\to+\infty}\to+\infty.$$

Apply such a result to $a_n:=\int_\S|\n u_n|^2$, $b_n:=\mathcal J_{\rho_0,\rho'_0}(u_n)$; then, $F(u):=f\left(\int_\S|\n u|^2\right)$ will satisfy the required properties.
\end{proof}

\

To extend the inequality to a wider range of parameters, we perform a blow--up analysis using results from the previous section.

\begin{proof}[Proof of (1) in Theorem \ref{mtteo}]
Suppose, by contradiction, that $\mathcal J_{\rho,\rho'}$ is not coercive for some $(\rho,\rho')$ satisfying $\rho<4\pi$, $\rho+\rho'<2\pi$.

In view of Lemma \ref{aperto}, the space of coercive parameters is given by $\mathring\L$, with $\L$ given by \eqref{lambda}, and it is not empty because of Lemma \ref{interpol}. Therefore there will be some $(\rho_0,\rho'_0)\in\de\L$ with $\rho_0<4\pi$ and $\rho_0+\rho'_0<2\pi$ due to Proposition \ref{nonmt}.

From Lemma \ref{aperto} we get that \eqref{coerc} does not hold true for $(\rho_0,\rho'_0)$, therefore $(\rho_0,\rho'_0)$ will satisfy \eqref{noncoerc} and we may apply Lemma \ref{perturb}.

Now, let us take $(\rho_n,\rho'_n)\in\mathring\L$ satisfying $(\rho_n,\rho'_n)\underset{n\to+\infty}\to(\rho_0,\rho'_0)$ and let us consider $\mathcal J_{\rho_n,\rho'_n}-F$. In view of \eqref{coerc} and the second assumption of \eqref{f}, the new functional will be coercive for any $(\rho_n,\rho'_n)$, as $\mathcal J$ is. Therefore, a sequence $u_n$ of minimizers satisfying

$$\left\{\begin{array}{ll}
-(1+F'(u_n))\Delta u_n+\frac{2(\rho_n+\rho'_n)}{|\S|}= 2\rho_n\frac{Ke^{u_n}}{\int_\S Ke^{u_n}}&\text{in }\S,\vspace{.3cm}\\
(1+F'(u_n))\de_\nu u_n=2\rho'_n\frac{he^\frac{u_n}2}{\int_{\de\S}he^\frac{u_n}2}&\text{on }\de\S.
\end{array}\right.$$

Since $1+F'(u_n)\underset{n\to+\infty}\to1$ uniformly, we are in a position to apply Theorem \ref{compa}. To get a contradiction and conclude the proof we suffice to exclude both alternatives in the theorem.

Alternative $(2)$ cannot hold true, because $(\rho_0,\rho'_0)$ does not belong to the critical set $\G$ defined in \eqref{critic}. On the other hand, if alternative $(1)$ held true, then $u_n\underset{n\to+\infty}\to u_0$ in $H^1(\S)$, with $u_0$ being a minimizer of $\mathcal J_{\rho_0,\rho'_0}-F$. But $\mathcal J_{\rho_0,\rho'_0}-F$ is unbounded from below by construction, therefore this is impossible.

We found the desired contradiction and therefore completed the proof.

\end{proof}

\

Finally, we can obtain a sharp inequality by performing a more accurate blow--up analysis.

\begin{proof}[Proof of (2) in Theorem \ref{mtteo}]
Fix $\rho'\in(-\infty,-2\pi)$ and take $\rho_n\underset{n\to+\infty}\nearrow4\pi$.

By the first part of Theorem \ref{mtteo}, $\mathcal J_{\rho_n,\rho'}$ is coercive, hence it admits a sequence of minimizers $u_n$. By writing, for any $u\in H^1(\S)$,

$$\mathcal J_{4\pi,\rho'}(u)=\lim_{n\to+\infty}\mathcal J_{\rho_n,\rho'}(u)\ge\lim_{n\to+\infty}\mathcal J_{\rho_n,\rho'}(u_n),$$
we see that we suffice to prove that the latter sequence is uniformly bounded from below.

To this purpose, we apply blow--up analysis from Theorem \ref{compa} to $u_n$ (to which in case we may add a suitable constant to verify \eqref{massmass}, which does not alter the values of $\mathcal J$).

In case alternative $(1)$ hold, then $\mathcal J_{\rho_n,\rho'}(u_n)$ is uniformly bounded because $u_n$ is compact in $H^1(\S)$. On the other hand, if blow--up occurs, then the blow--up set $\mathcal S$ is disjointed with $\de\S$, otherwise we would have

$$4\pi\#(\mathcal S\cap\de\S)\underset{n\to+\infty}\leftarrow2(\rho_n+\rho')\le8\pi+2\rho'<4\pi.$$

Therefore, $u_n$ is bounded from above on $\de\S$, hence one may use inequality \eqref{chygineq} to get

$$\mathcal J_{\rho_n,\rho'}(u_n)\ge\int_\S\left(\frac12|\n u_n|^2+\frac{2\rho_n}{|\S|}u_n\right)-2\rho_n\log\int_\S e^{u_n}-C\ge-C.$$

\

Now, fix $\rho\in(-\infty,4\pi)\setminus\{0,2\pi\}$ and take $(\rho_n,\rho'_n)\underset{n\to+\infty}\nearrow(\rho,2\pi-\rho)$.

As in the previous case, we will suffice to show $\mathcal J_{\rho_n,\rho'_n}(u_n)\ge-C$ for some sequence $u_n$ of minimizers to $\mathcal J_{\rho_n,\rho'_n}$. Again, we apply Theorem \ref{compa} to $u_n$.

If the first alternative occurs in the theorem, then boundedness of energy is immediate. On the other hand, if $u_n$ blows up, then $\mathcal S\cap\mathring\S=\emptyset$, since otherwise we would get $2\rho_n\ge8\pi$, therefore $\mathcal S\cap\de\S\ne\emptyset$.

We are then in position to apply Proposition \ref{quasisharp} to get a mutual control between the two nonlinear term in $\mathcal J$. From this, we obtain:

$$\mathcal J_{\rho_n,\rho'}(u_n)\ge\int_\S\left(\frac12|\n u_n|^2+\frac{2(\rho_n+\rho'_n)}{|\S|}u_n\right)-4(\rho_n+\rho'_n)\log\int_{\de\S}e^\frac{u_n}2-C.$$

Now, we can pass from the mean values on $\S$ to the ones on $\de\S$ by subtracting from $u$ the solution $\vfi$ to \eqref{phi}, as in the proof of Lemma \ref{interpol}; then, one concludes by applying inequality \eqref{liliuineq}.
\end{proof}

\

\begin{remark}
The only case which is not covered by Theorem \ref{mtteo} is $(\rho,\rho')=(4\pi,-2\pi)$: we showed that $\mathcal J_{4\pi,-2\pi}$ is not coercive, but not whether it is bounded from below.

In fact, the proof of the theorem relies on Proposition \ref{quasisharp}, which does not hold for this case, as pointed out in Remark \ref{remark}.
\end{remark}

\

\section{Existence of min--max solutions}
\setcounter{equation}0

In this subsection we prove Theorem \ref{eximinmax} in the cases where the energy functional is not bounded from below, namely where $N,M$ satisfy \eqref{nm}.

The argument consists in \emph{comparing} very low energy sublevels $\mathcal J^{-L}$, defined in \eqref{jl}, with some spaces of barycenters, whose homology is well--known. In particular, the homology groups of $\mathcal J^{-L}$ will contain a copy of the homology groups of such barycenters, which under some assumptions is not trivial. Therefore, low energy sublevels are not contractible and existence of solutions will follow.

\

The barycenters on a metric space $\M$ is given by finitely--supported probability measures on $\M$, namely convex combinations of Dirac deltas, equipped with the $\mathrm{Lip'}$ topology:

$$(\M)_K=\left\{\sum_{i=1}^Kt_i\d_{x_i}:\,\sum_{i=1}^Kt_i=1,\,x_i\in\M\right\}\quad\quad\quad\|\mu\|_{\mathrm{Lip}'(\M)}:=\sup_{\vfi\in\mathrm{Lip}(\M),\|\vfi\|_{\mathrm{Lip}(M)}\le1}\left|\int_\M\vfi\mathrm d\mu\right|.$$

We will take as $\M$ either $\de\S$ or a compact subset $\widetilde\S\Subset\S$ well--separated from the boundary, namely

\beq\label{sigmaw}
\widetilde\S:=\S\setminus(\de\S)^\d=\left\{x\in\S:\,\dist(x,\de\S)\ge\d\right\},
\eeq
with $\d>0$ small enough so that $(\de\S)^\d$ it is a deformation retract of $\de\S$ and $\widetilde\S$ is a deformation retract of $\S$. In particular, taking $M,N$ as in the statement of Theorem \ref{eximinmax}, we will consider the space $\mathcal X$ defined by

\beq\label{x}
\boxed{\mathcal X:=\left\{\begin{array}{ll}\left(\widetilde\S\right)_N&\text{if }N\ge M,\vspace{.3cm}\\(\de\S)_M&\text{if }N<M.\end{array}\right.}
\eeq

Roughly speaking, if $N$, is greater, then the interior of $\S$ plays a more important role than the boundary. This is consistent with the blow--up analysis from Theorem \ref{compa}, where interior blow--up depends on $\rho$ and boundary blow--up depends on the sum $\rho+\rho'$.

In the case $N<M$ we always get a noncontractible space, as $\de\S$ is always noncontractible. If $N\ge M$, $\mathcal X$ is contractible if $\S$ is, therefore in this last case we do not get existence of solutions, as shown by Figure 3.

\begin{figure}[h!]
\centering
\includegraphics[width=.5\linewidth]{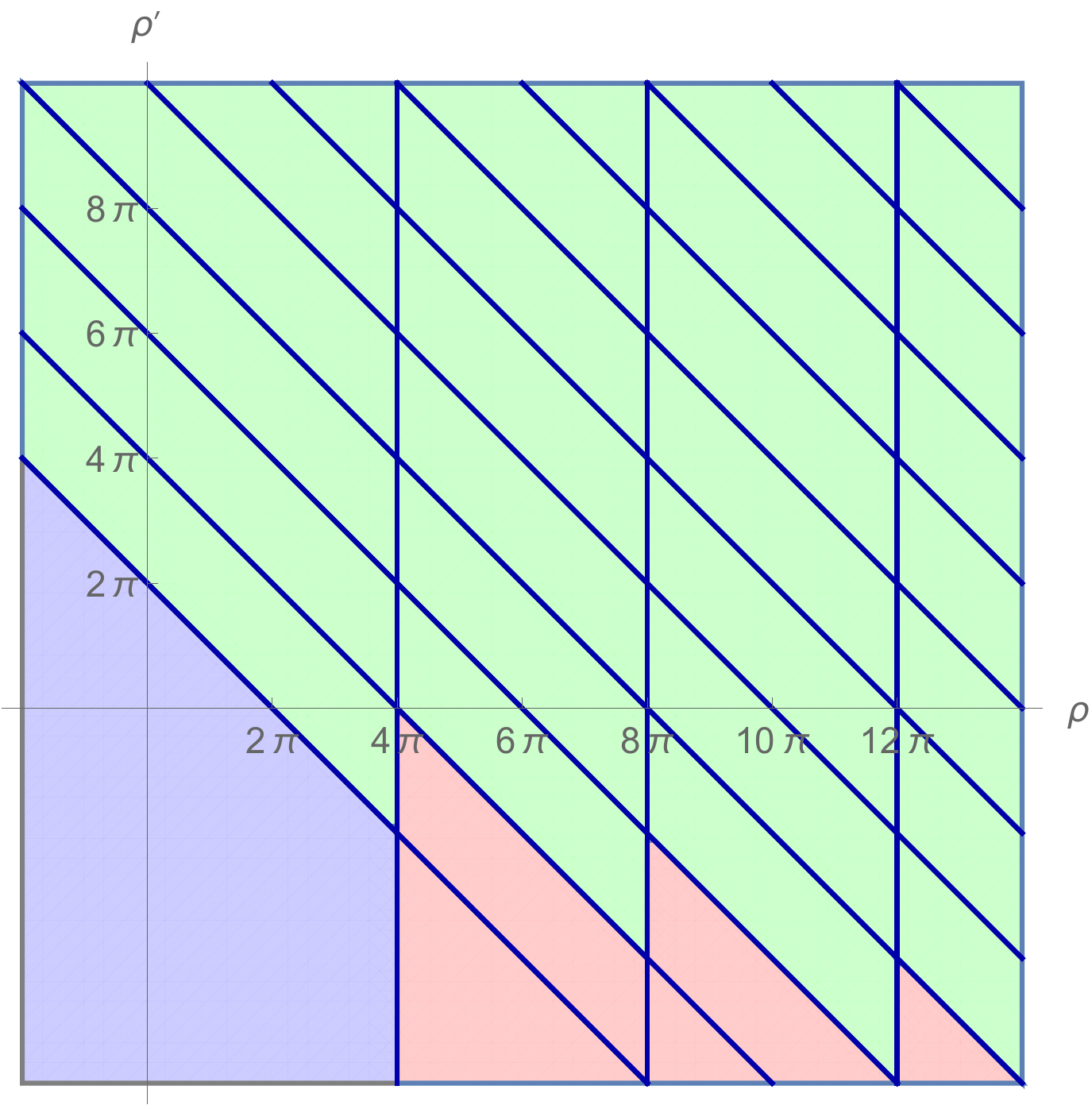}
\caption{The regions where \eqref{meanfield} has solutions {\color{blue}on} {\color[rgb]{0,.5,0}any} $\S$ or on {\color{red}multiply connected} $\S$}
\end{figure}

\begin{proposition}\label{phipsi}
Let $\mathcal J^{-L}$ be defined by \eqref{jl} and $\mathcal X$ be defined by \eqref{x}. Then, for $L\gg0$ large enough there exist two maps

$$\Phi:\mathcal X\to\mathcal J^{-L},\quad\quad\quad\quad\quad\quad\Psi:\mathcal J^{-L}\to\mathcal X,$$
such that $\Psi\circ\Phi:\mathcal X\to\mathcal X$ is homotopically equivalent to the identity map.
\end{proposition}

From this it is standard to deduce Theorem \ref{eximinmax}, using compactness results from Theorem \ref{compa} and a deformation lemma from \cite{luc}.

\begin{proof}[Proof of Theorem \ref{eximinmax}]
From Proposition \ref{phipsi} and the properties of homology one gets $H_*(\mathcal X)\hookrightarrow H_*\left(\mathcal J^{-L}\right)$. By assuming either $N<M$ or $\S$ being multiply connected we get that $\mathcal X$ has some nontrivial homology groups (see \cite{lucaks}, Lemma 3.3 in the former case and \cite{bdm}, Proposition 3.2 in the latter case); therefore, $\mathcal J^{-L}$ also has non-trivial homology, hence is not contractible.

On the other hand, since $(\rho,\rho')\notin\G$, then Corollary \ref{notingamma} holds true. Therefore, one can argue a deformation argument as in \cite{luc}, Proposition 1.1 (see also \cite{str-mon}) to deduce that $\mathcal J^L$ is a deformation retract of the whole space $H^1(\S)$ if $L\gg0$ is large enough, hence it is contractible.

Therefore, $\mathcal J^L$ and $\mathcal J^{-L}$ are not homotopically equivalent. Arguing again like in \cite{luc}, Proposition 1.1, there must be a sequence $u_n$ of solutions to \eqref{meanfieldn} with $K=K_n,h=h_n$ satisfying $-L\le\mathcal J_{\rho_n,\rho'_n}(u_n)\le L$ and $(\rho_n,\rho'_n)\underset{n\to+\infty}\to(\rho,\rho')$. By Theorem \ref{compa} and the assumptions on $\rho,\rho'$, we have $u_n\underset{n\to+\infty}\to u$, the latter being a solution to \eqref{meanfield}.
\end{proof}

\begin{remark} One can similarly deduce multiplicity of solutions to \eqref{meanfield} via Morse theory, as the homology of $\mathcal X$ has been explicitly computed in \cite{bdm,dmls,lucaks} and in particular

\begin{eqnarray*}
N\ge M\quad&\Rightarrow&\quad H_q(\mathcal X)=\left\{\begin{array}{ll}\Z^{N+g-1\choose g-1}&\text{if }q=2N-1,\vspace{.3cm}\\0&\text{if }q\ne2N-1;\vspace{.3cm}\end{array}\right.\\
N<M\quad&\Rightarrow&\quad H_q(\mathcal X)=\left\{\begin{array}{ll}\Z^{{q-M+g+1\choose g}{g\choose2M-q-1}}&\text{if }\max\{M-1,2M-g-1\}\le q\le2M-1,\vspace{.3cm}\\0&\text{if }q<\max\{M-1,2M-g-1\}\text{ or }q>2M-1;\end{array}\right.
\end{eqnarray*}
where $g$ is the genus of $\S$.

Multiplicity will hold true only if $\mathcal J$ is a Morse function; however, one can prove that this holds true for a \emph{generic} choice of the potentials $K,h$ and the metric $g$. For further details see \cite{fra,bdm}.
\end{remark}

\

The rest of this section will be devoted to the proof of Proposition \ref{phipsi}.

We start with the construction of the map $\Phi$, consisting in a family of test function modeled on $\mathcal X$ whose energy is arbitrarily low; this is is a generalization of the construction made in Proposition \ref{nonmt}. Since such test functions are very well--studied in Liouville--type problems, the proof will be sketchy.

\begin{lemma}\label{test}
For any $\l>0$ define $\Phi^\l:\mathcal X\to H^1(\S)$ as

\begin{eqnarray*}
\mathcal X&\overset{\Phi^\l}\longrightarrow&H^1(\S)\\
\xi=\sum_it_i\d_{x_i}&\mapsto&\phi^\l_\xi(x)=\log\sum_it_i\left(\frac{\l}{1+\l^2\dist^2(x,x_i)}\right)^2.
\end{eqnarray*}

Then, $\mathcal J_{\rho,\rho'}\left(\phi^\l_\xi\right)\underset{\l\to+\infty}\to-\infty$ uniformly on $\mathcal X$.
\end{lemma}

\begin{proof}
The following estimates can be easily verified (see for instance \cite{dm,guoliu}):

\begin{eqnarray*}
\int_\S\left|\n\phi^\l_\xi\right|^2&=&\left\{\begin{array}{ll}32N\pi\log\l+O(1)&\text{if }N\ge M,\vspace{.3cm}\\16M\pi\log\l+O(1)&\text{if }N<M;\end{array}\right.\vspace{.3cm}\\
\log\int_\S Ke^{\phi^\l_\xi}&=&O(1);\vspace{.3cm}\\
\log\int_{\de\S}he^\frac{\phi^\l_\xi}2&=&\left\{\begin{array}{ll}-\log\l+O(1)&\text{if }N\ge M,\vspace{.3cm}\\O(1)&\text{if }N<M;\end{array}\right.\vspace{.3cm}\\
\int_\S\phi^\l_\xi&=&-2|\S|\log\l+O(1);
\end{eqnarray*}
with $O(1)$ independent of $\xi$. Therefore, one gets

$$\mathcal J\left(\phi^\l_\xi\right)=\left\{\begin{array}{ll}(16N\pi-4\rho)\log\l+O(1)&\text{if }N\ge M,\vspace{.3cm}\\(8M\pi-4(\rho+\rho'))\log\l+O(1)&\text{if }N<M,\end{array}\right.$$
which in both cases goes uniformly to $-\infty$ as $\l\to+\infty$.
\end{proof}

\

Let us now consider the map $\Psi:\mathcal J^{-L}\to\mathcal X$. Its existence will follow by showing that for any $u\in\mathcal J^{-L}$ the unit measure $\frac{Ke^u}{\int_\S Ke^u}$ is close to a barycenter of $\mathcal X$.

The first step is an \emph{improved} Moser--Trudinger inequality: roughly speaking, if $u$ is \emph{spread} in some different regions of $\S$, then the constants $4\pi,-2\pi$ in inequality \eqref{mt} can be multiplied by some integer numbers, depending on the number and the position of such regions. In order to do it, the following localized version of Moser--Trudinger type inequalities will be of use (see \cite{ruizsoriano}, Proposition 2.2 for a proof).

\begin{proposition}\label{mt1l}
Let $\e>0,\d>0$ and $\Omega_1\subset\S$ such that $\dist(\Omega_1,\de\S)>\d$.
Then, there exists $C=C_{\e,\d}$ such that for every $u\in H^1(\S)$,

$$16\pi\log\int_{\Omega_1}e^u\le\int_{\Omega^\d_1}\left|\n u\right|^2+\e\int_\S\left|\n u\right|^2+\frac{16\pi}{|\S|}\int_\S u+C.$$

\end{proposition}

We now extend this result to the case when $\Omega$ may touch the boundary, following \cite{chen-li}, Theorem 2.1.

\begin{proposition}\label{mt2l}
Let $\e>0,\d>0$ and $\G_2:=\Omega_2\cap\de\S\neq\emptyset$.
Then, there exists $C=C_{\e,\d}$ such that for every $u\in H^1(\S)$

$$16\pi\log\int_{\Omega_2}e^u-16\pi\log\int_{(\G_2)^\frac\d2}e^\frac u2\le\int_{\Omega^\d_2}\left|\n u\right|^2+\frac{8\pi}{|\S|}\int_\S u+\e\int_\S\left|\n u\right|^2+C.$$

\end{proposition}

\begin{proof}
First of all, as the inequality we want to prove is invariant by addition of constants, it will not restrictive to prove it for $u\in\overline H^1(\S)$.

Consider a smooth cut--off function $g$ verifying

$$g(x)=\left\{\begin{array}{ll}
1&\text{if }x\in\Omega_2,\\
0&\text{if }x\in\S\setminus(\Omega_2)^\frac\d2.
\end{array}\right.$$

Clearly $gu\in H^1(\S)$, and moreover

\begin{eqnarray*}
16\pi\left(\log\int_{\Omega_2}e^u-\log\int_{(\G_2)^\frac\d2}e^\frac u2\right)&\le&16\pi\left(\log\int_\S e^{gu}-\log\int_{\de\S}e^\frac{gu}2\right)\\
&\le&(1+\e)\int_\S|\n(gu)|^2+\frac{8\pi}{|\S|}\int_\S gu+C,
\end{eqnarray*}
by inequality \eqref{mt} with $\rho=\frac{4\pi}{1+\e}$. Since $\int_\S u=0$, Poincar\'e--Wirtinger inequality gives

$$\int_\S gu\le\int_\S|u|\le C\sqrt{\int_\S|\n u|^2}\le\e\int_\S|\n u|^2+C.$$

On the other hand, applying the product rule and Cauchy inequality one has

$$\int_\S|\n(gu)|^2\le\int_\S\left((1+\e)g^2|\n u|^2+\left(1+\frac1\e\right)u^2|\n g|^2\right)\le\int_{(\Omega_2)^\d}|\n u|^2+\e\int_\S|\n u|^2+C_{\e,\d}\int_\S u^2.$$

Now, for any $a>0$, we define $\eta:=\left|\{x\in\S:u(x)\ge a\}\right|$ and apply the previous computations to $\left(u-a\right)^+$: we obtain

\begin{eqnarray}\label{cuenta2}
\nonumber16\pi\left(\log\int_{\Omega_2}e^u-\log\int_{(\G_2)^\frac\d2}e^\frac u2\right)&\le&16\pi\left(\log\left(e^a\int_{\Omega_2}e^{(u-a)^+}\right)-\log e^\frac a2\int_{(\G_2)^\frac\d2}e^\frac{(u-a)^+}2\right)\vspace{.3cm}\\
&\le&8\pi a+\int_{(\Omega_2)^\d}\left|\n u\right|^2+\e\int_\S\left|\n u\right|^2+C_{\e,\d}\int_\S\left((u-a)^+\right)^2.
\end{eqnarray}

To deal with the first term, we use Poincar\'e-Wirtinger and Cauchy inequalities:

\beq\label{mtl6}
a\le\frac1\eta\int_{\{u\ge a\}}u\le\frac1\eta\int_\S\left|u\right|\le\frac C\eta\sqrt{\int_\S\left|\n u\right|^2}\le\e\int_\S\left|\n u\right|^2+\frac C{4\e\eta^2}.
\eeq

On the other hand, the last term in \eqref{cuenta2} can be estimated using H\"older and Sobolev inequalities:

\beq\label{mtl4}
\int_\S\left((u-a)^+\right)^2\le\sqrt\eta\sqrt{\int_\S\left((u-a)^+\right)^4}\le C\sqrt\eta\int_\S\left(\left((u-a)^+\right)^2+|\n u|^2\right).
\eeq

If we choose $a$ so large that $\eta\le\e^2$ and $\e$ is small enough, then \eqref{mtl4} gives $\int_\S\left((u-a)^+\right)^2\le C'\e\int_\S\left|\n u\right|^2$. This estimate, together with \eqref{cuenta2} and \eqref{mtl4}, concludes the proof.

\end{proof}

We can now state and prove the following improved Moser--Trudinger inequality.

\begin{lemma}\label{mtimproved}Assume $\Omega_{11},\dots,\Omega_{1J},\Omega_{21},\dots,\Omega_{2K}\subset\S$ and $u\in H^1(\S)$ satisfy

\begin{eqnarray}\label{masa}
\nonumber\dist(\Omega_{ij},\Omega_{kl})\ge\d,&\quad&\forall(i,j)\ne(k,l);\vspace{0.3cm}\\[5pt]
\nonumber\dist(\Omega_{1j},\de\S)\ge\d,&\quad&\forall j=1,\dots,J;\vspace{0.3cm}\\[5pt]
\frac{\int_{\Omega_{ij}}e^u}{\int_\S e^u}\ge\d,&\quad&\forall i,j.
\end{eqnarray}

Then, for any $\e>0$ there exists $C=C_\e$ such that

$$16(J+K)\pi\log\int_\S e^u-16K\pi\log\int_{\de\S}e^\frac u2\le(1+\e)\int_\S\left|\n u\right|^2+\frac{8\pi(2J+K)}{|\S|}\int_\S u+C.$$

\end{lemma}

\begin{proof}
First take $\Omega_{11},\ldots,\Omega_{1J}$ verifying \eqref{masa} and apply Proposition \ref{mt1l} for each one, then

$$16\pi\log\int_\S e^u\le16\pi\log\left(\frac1\d\int_{\Omega_{1j}}e^u\right)\le\int_{(\Omega_{1j})^\d}\left|\n u\right|^2+\e\int_\S\left|\n u\right|^2+\frac{16\pi}{|\S|}\int_\S u+C_{\e,\d},$$
for $j=1,\ldots,J$. Now, take $\Omega_{21},\ldots,\Omega_{2K}$ verifying \eqref{masa} and apply Proposition \ref{mt2l}, so

\begin{eqnarray*}
16\pi\left(\log\int_\S e^u-\log\int_{\de\S}e^\frac u2\right)&\le&16\pi\log\left(\frac1\d\int_{\Omega_{2k}}e^u\right)-\log\int_{\G_{2k}}e^\frac u2\\
&\le&\int_{(\Omega_{2k})^\d}\left|\n u\right|^2+\e\int_\S\left|\n u\right|^2+\frac{8\pi}{|\S|}\int_\S u+C_{\e,\d},
\end{eqnarray*}
for $k=1,\ldots,K$. Finally, we sum both previous inequalities for all $j,k$ to get

\begin{eqnarray*}
&&16(J+K)\pi\log\int_\S e^u-16K\pi\log\int_{\de\S}e^\frac u2\\
&\le&\int_{\bigcup_{i,j}(\Omega_{ij})^\d}\left|\n u\right|^2+(J+K)\e\int_\S\left|\n u\right|^2+\frac{8\pi(2J+K)}{|\S|}\int_\S u+C_{\e,\d}\\
&\le&(1+(J+K)\e)\int_\S\left|\n u\right|^2+\frac{8\pi(2J+K)}{|\S|}\int_\S u+C_{\e,\d},
\end{eqnarray*}
which concludes the proof.
\end{proof}

\

In order to apply Lemma \ref{mtimproved} to a wider range of parameter, we need an estimate of the boundary nonlinear term by means of the interior nonlinear term. Such a result somehow contains important information about the relation between $\rho$ and $\rho'$ in $\mathcal J$.

The following proposition gives a sort of \emph{monotonicity} property to the energy $\mathcal J$, not only with respect to each parameter $\rho,\rho'$, as shown in the proof of Lemma \ref{aperto}, but also with respect to the sum $\rho+\rho'$.

\begin{proposition}\label{prop0}
For any $\e>0$ there exists a constant $C=C_\e$ such that for $u\in H^1(\S)$,

\beq\label{prop01}
\log\int_{\de\S}e^\frac u2\le\frac12\log\int_\S e^u+\e\int_\S\left|\n u\right|^2+C.
\eeq

\end{proposition}
\begin{proof}
Take a $C^\infty$ vector field $\eta$ whose restriction to $\de\S$ is the outward normal vector field. By the Stokes' Theorem, we obtain that:

$$\int_{\de\S}e^\frac u2=\int_\S\mathrm{div}\left(\eta e^\frac u2\right)=\int_\S\left(\mathrm{div}(\eta)e^\frac u2+\eta\frac{\n u}2e^\frac u2\right).$$

By the smoothness of the vector field and H\"older inequality we have

$$\int_\S\mathrm{div}(\eta)e^\frac u2\le C\left(\int_\S e^u\right)^\frac12,$$

$$\int_\S\eta\frac{\n u}2e^\frac u2\le C\int_\S e^\frac u2|\n u|\le C\left(\int_\S e^u\right)^\frac12\left(\int_\S|\n u|^2\right)^\frac12.$$

Therefore,

\beq\label{prop02}
\int_{\de\S}e^\frac u2\le C\left(\int_\S e^u\right)^\frac12\left(1+\left(\int_\S\left|\n u\right|^2\right)^\frac12\right)
\eeq

Taking logarithms in \eqref{prop02} we obtain

$$\log\int_{\de\S}e^\frac u2\le\frac12\log\int_\S e^u+\log\left(1+\left(\int_\S\left|\n u\right|^2\right)^\frac12\right)+C.$$

Using the fact that $\log(1+t)\le t$ for any $t\ge0$ and the general Cauchy inequality, we obtain \eqref{prop01} as we wanted. 
\end{proof}

\

\begin{corollary}\label{mtimpr2}
Under the assumptions of Lemma \ref{mtimproved}, $\mathcal J_{\rho,\rho'}$ is bounded from below for any $\rho,\rho'$ satisfying $\rho<4\pi(J+K)$, $\rho+\rho'<2\pi(2J+K)$.
\end{corollary}

\begin{proof}
We apply Lemma \ref{mtimproved} with $\e>0$ small enough to have

$$\frac{8\pi(J+K)}{1+2\e}-2\rho\ge0,\quad\quad\quad\quad\quad\quad\frac{8\pi(2J+K)}{1+2\e}-4\rho-4\rho'\ge0.$$

Next, we apply Proposition \ref{prop0} with some $\e'$, possibly different from $\e$, to be chosen later. Notice also that, due to Jensen and trace Sobolev inequalities, we have

$$\log\int_{\de\S}e^\frac u2\ge\frac1{|\de\S|}\int_{\de\S}u+\log|\de\S|\ge\frac1{|\S|}\int_{\S}u-C\sqrt{\int_\S|\n u|^2}+\log|\de\S|\ge\frac1{|\S|}\int_{\S}u-\e'\int_\S|\n u|^2-C.$$

Therefore, assuming without loss of generality $\int_\S u=0$, we get

\begin{eqnarray*}
\mathcal J_{\rho,\rho'}(u)&\ge&\frac{1+\e}{2+4\e}\int_\S|\n u|^2-2\rho\log\int_\S e^u-4\rho'\log\int_{\de\S}e^\frac u2+\frac\e{2+4\e}\int_\S|\n u|^2\\
&\ge&\left(\frac{8\pi(J+K)}{1+2\e}-2\rho\right)\log\int_\S e^u-\left(\frac{8\pi K}{1+2\e}+4\rho'\right)\log\int_{\de\S}e^\frac u2+\frac\e{2+4\e}\int_\S|\n u|^2-C\\
&\ge&\left(\frac{8\pi(2J+K)}{1+2\e}-4\rho-4\rho'\right)\log\int_{\de\S}e^\frac u2-\e'\left(\frac{8\pi(J+K)}{1+2\e}-2\rho\right)\int_\S|\n u|^2+\frac\e{2+4\e}\int_\S|\n u|^2-C\\
&\ge&-\e'\left(\frac{8\pi(2J+K)}{1+2\e}-4\rho-4\rho'+\frac{8\pi(J+K)}{1+2\e}-2\rho\right)\int_\S|\n u|^2+\frac\e{2+4\e}\int_\S|\n u|^2-C
\end{eqnarray*}
which is bounded from below if $\e'$ is chosen properly.
\end{proof}

\

Thanks to Lemma \ref{mtimproved}, we can show that, as $\mathcal J_{\rho,\rho'}(u)$ is more negative, $u$ is closer to some barycenter space. However, this space will be not be the $\mathcal X$ introduced before, but rather a larger $\mathcal X'$, containing more than one \emph{stratum} of barycenters centered at both $\mathring\S$ and $\de\S$. In particular, $\mathcal X'$ will contain combinations of $J$ points in $\mathring\S$ and $K$ points in $\de\S$ satisfying either $J+K\le N$ or $2J+K\le M$, namely:

\beq\label{xprim}
\boxed{\mathcal X':=\left\{\begin{array}{ll}(\S)_N=\bigcup_{i=0}^N(\S)_{i,N-i}&\text{if }N\ge M,\vspace{.3cm}\\
\bigcup_{i=0}^{M-N+1}(\S)_{i,M-2i}\cup\bigcup_{i=M-N}^N(\S)_{i,N-i}&\text{if }N<M;\end{array}\right.}
\eeq
where

\begin{equation}
\label{sigmajk}
(\S)_{J,K}:=\left\{\sum_{i=1}^Jt_i\d_{p_i}+\sum_{i=1}^Ks_i\d_{q_i};\,p_i\in\mathring\S,\,q_i\in\de\S,\,\sum_{i=1}^Jt_i+\sum_{i=1}^Ks_i=1\right\}.
\end{equation}

Notice that the explicit expression of $\mathcal X'$, as well as $\mathcal X$, is different in the cases $N\ge M$ and $N<M$, hinting that the two cases could be somehow different.

\begin{lemma}\label{lemma2}
For any $\e>0$ there exists $L=L_\e>0$ such that any $u\in\mathcal J^{-L}$ satisfies

$$\dist_{\mathrm{Lip}'(\S)}\left(\frac{Ke^u}{\int_\S Ke^u},\mathcal X'\right)<\e,$$
where $\mathcal X'$ is defined in \eqref{xprim}.
\end{lemma}

\begin{proof}

We follow a rather widely-used scheme from \cite{bjmr} (see also \cite{dm,lucaks}), therefore we will be sketchy.

We suffice to show the following fact: for any $\e>0$ there exists $L=L_\e>0$ and $x_{11},\dots,x_{1J}\in\S$, $x_{21},\dots,x_{2K}\in\de\S$ such that

\beq\label{claim}
\begin{array}{ll}\text{either}&J+K\le N\vspace{.3cm}\\\text{or}&2J+K\le M\end{array}\quad\quad\quad\text{and}\quad\quad\quad\frac{\int_{\bigcup_{i,j}B_\e(x_{ij})}Ke^u}{\int_\S Ke^u}\ge1-\e.
\eeq
In fact, arguing as in \cite{bjmr}, Proposition 4.6, there exist $t_{11},\dots,t_{1J},t_{21},\dots,t_{2K}\in[0,1]$ such that $\sum_{i,j}t_{ij}=1$ and

$$\xi:=\sum_{i,j}t_{ij}\d_{x_{ij}}\quad\quad\quad\Rightarrow\quad\quad\quad\dist_{\mathrm{Lip}'(\S)}\left(\frac{Ke^u}{\int_\S Ke^u},\xi\right)<\e;$$
notice that, due to the algebraic assumptions in \eqref{claim} and the definition of $\mathcal X'$, we have $\xi\in\mathcal X'$.

To show the claim \eqref{claim}, we argue by contradiction. If $\frac{\int_{\bigcup_{i,j}B_\e(x_{ij})}Ke^u}{\int_\S Ke^u}<1-\e$ for any such $x_{ij}$, then we can apply a covering argument (\cite{bjmr}, Lemma 4.4; \cite{lucaks}, Lemma 3.16 and minor modifications) to get the following: there exist $\d(\e)>0$ and $x'_{11},x'_{1J'},x'_{21},x'_{2K'}$ with $J'+K'\ge N+1$, $2J'+K'\ge M+1$ and

\begin{eqnarray*}
\dist\left(x'_{ij},x'_{kl}\right)\ge\d,&\quad&\forall(i,j)\ne(k,l);\vspace{0.2cm}\\[5pt]
\dist\left(x'_{1j},\de\S\right)\ge\d,&\quad&\forall j=1,\dots,J';\vspace{0.2cm}\\[5pt]
\frac{\int_{B_\d(x_{ij})}Ke^u}{\int_\S Ke^u}\ge\d,&\quad&\forall i,j.
\end{eqnarray*}

We are now in position to apply Lemma \ref{mtimproved} with $\Omega_{ij}:=B_\d\left(x'_{ij}\right)$ which gives, together with Corollary \ref{mtimpr2}, $\mathcal J(u)\ge-C$. This proves the claim and the present proposition.

\end{proof}

\

At this point, we need to fill the gap between the spaces $\mathcal X$ and $\mathcal X'$. Actually, we show that the latter retracts on the former.

This is a crucial step in the proof of Theorem \ref{eximinmax}, as in some cases existence of min--max solutions may fail if one only has maps $\Phi:\mathcal X\to\mathcal J^{-L}$ and $\Psi:\mathcal J^{-L}\to\mathcal X'$ without any relation between $\mathcal X$ and $\mathcal X'$.

More precisely, in the case $N\ge M$, we have an actual deformation retract, namely there is no topological loss when passing from $\mathcal X'$ to the simpler $\mathcal X$. In particular, in the case of a simply connected $\S$, which is not covered by Theorem \ref{eximinmax}, not only $\mathcal X$ but also $\mathcal X'$ is contractible.

\begin{proposition}\label{propretract}
There exists a retraction $\Pi:\mathcal X'\to\mathcal X$ such that $\Pi(\xi)=\xi$ for any $\xi\in\mathcal X$.

If $N\ge M$, such a map can be taken as a deformation retract.
\end{proposition}

\begin{proof}

We divide the case $N\ge M$ and $N<M$.

If $N\ge M$ we consider a deformation retract $\pi:\S\to\widetilde\S$ (see \eqref{sigmaw} for the definition of $\widetilde\S$) and we extend it to $\mathcal X'$ via push-forward, namely applying $\pi$ to any point of the support of $\xi\in\mathcal X'$:

\begin{eqnarray*}
\mathcal X'&\overset\Pi\longrightarrow&\mathcal X\\
\xi=\sum_it_i\d_{x_i}&\mapsto&\sum_it_i\d_{\pi(x_i)}.
\end{eqnarray*}

$\Pi$ is well-defined and continuous, because in the case $N\ge M$ we have $\mathcal X'=(\S)_N$ (see \eqref{xprim} for the definition). Moreover, it is a retraction because $\Pi(\xi)=\xi$ if $\xi\in\mathcal X=\left(\widetilde\S\right)_N$.

Furthermore, $\Pi$ is a deformation retract between $\mathcal X'$ and $\mathcal X$ because if $h:\S\times[0,1]\to\S$ is a homotopical equivalence with $h(\cdot,0)=\mathrm{Id}_\S$ and $h(\cdot,1)=\pi$, then

$$H(t,\xi):=\sum_it_i\d_{x_i}\mapsto\sum_it_i\d_{h(x_i,t)}$$
is a homotopical equivalence on $\mathcal X'$ with $H(\cdot,0)=\mathrm{Id}_{\mathcal X'}$ and $H(\cdot,1)=\Pi$.

\

Let us now consider the case $N<M$.

This time we will map $\S$ onto a cone in the space $(\de\S)_2$ of barycenters centered at two points on the boundary, then we will extend the map to $\mathcal X'$ via push-forward.

Take $\widetilde\S$ as before and $\Omega:=\overline{\S\setminus\widetilde\S}$. If $\widetilde\S$ is chosen properly, then $\Omega$ is homeomorphic to $\de\S\times[0,1]$, with $\de\S$ corresponding to $\de\S\times\{1\}$ and $\de\widetilde\S$ corresponding to $\de\S\times\{0\}$. Now, we construct $\pi$ in such a way that $\de\S$ is fixed and the whole $\widetilde\S$ is identified with a given $x_0\in\de\S$; to properly \emph{glue} the two conditions, we exploit the identification between $\Omega$ and $\de\S\times[0,1]$ to linearly interpolate between the deltas centered at $x_0$ and at some $x\in\de\S$:

\begin{eqnarray*}
\S&\overset\pi\longrightarrow&(\de\S)_2\vspace{.3cm}\\
x&\mapsto&\left\{\begin{array}{ll}\d_{x_0}&\text{if }x\in\widetilde\S,\vspace{.3cm}\\(1-t)\d_{x_0}+t\d_y&\text{if }x=(y,t)\in\Omega,\vspace{.3cm}\\\d_x&\text{if }x=(x,1)\in\de\S.\end{array}\right.
\end{eqnarray*}

It is clear that $\pi$ is well-defined and, due to the homeomorphism between $\Omega$ and $\de\S\times[0,1]$, it is continuous.

Figure 4 shows how the retraction behaves depending on the location of $x\in\S$.

\begin{figure}[h!]
\centering
\includegraphics[width=\linewidth]{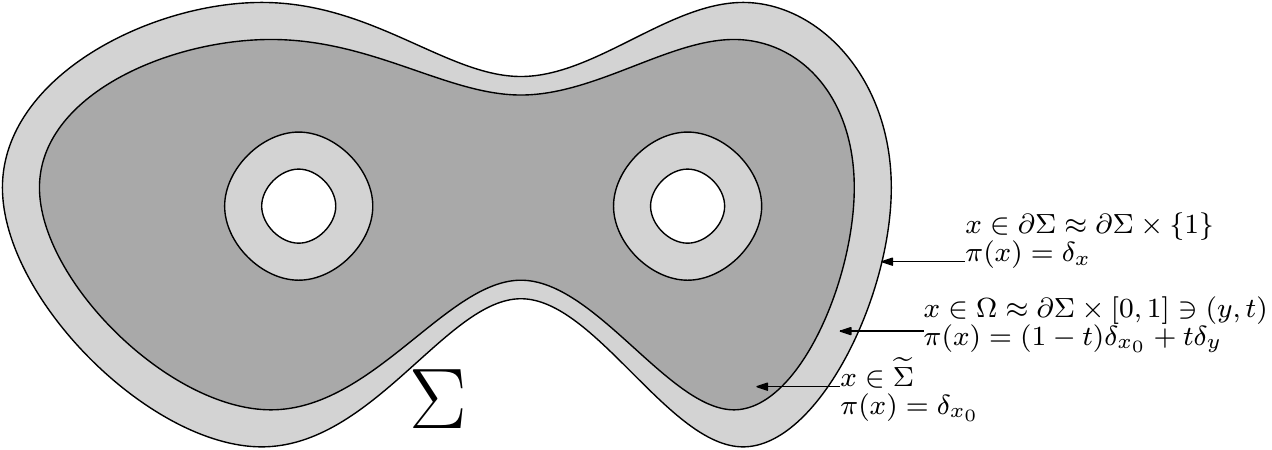}
\caption{The map $\pi:\S\to(\de\S)_2$}
\end{figure}

To extend $\pi$ to some $\Pi$ defined on the whole $\mathcal X=(\de\S)_M$ we set:

\begin{eqnarray*}
\mathcal X'&\overset\Pi\longrightarrow&\mathcal X\\
\xi=\sum_it_i\d_{x_i}&\mapsto&\sum_it_i\d_{\pi(x_i)}.
\end{eqnarray*}

First of all, since $\pi(x)=x$ for any $x\in\de\S$, then $\Pi|_{\mathcal X}=\mathrm{Id}_{\mathcal X}$.

Now, let us show that $\Pi(\xi)\in\mathcal X$ also for $\xi\in\mathcal X'\setminus \mathcal X$: in this case, from the definition of $\mathcal X$, $\xi$ will be supported in at most $M-1$ points. Therefore, since $\supp(\Pi(\xi))=(\supp(\xi)\cap\de\S)\cup\{x_0\}$, $\Pi(\xi)$ is supported in at most $M$ points in $\de\S$, namely it belongs to $\mathcal X$.

Finally, from the definition we also get that $\Pi$ coincides on the intersections of different strata of $\mathcal X$ and it is continuous, therefore it is the desired retraction.
\end{proof}

\

We have now all the tools to get the proof of Proposition \ref{phipsi}.

\begin{proof}[Proof of Proposition \ref{phipsi}]
We take $\Phi:=\Phi^{\l_0}$ as in Lemma \ref{test}, with $\l_0=\l_0(L)$ so large that $\Phi(\xi)\le-L$ for any $\xi\in\mathcal X$. Lemma \ref{test} ensures that this can be done for any $L>0$, which will be chosen later.

\

As for $\Psi$, we exploit Lemma \ref{lemma2} and the fact that $\mathcal X'$ is an Euclidean deformation retract, as each stratum is a Euclidean deformation retract. One can prove the latter fact by arguing as in \cite{bjmr}, Appendix A, or show the former fact by adapting the proof of \cite{lucaks}, Lemma 3.1.

Therefore, for $\e_0>0$ small enough we have a projection

$$\pi:\left\{\mu\in\mathrm M(\S):\,\dist_{\mathrm{Lip}'}(\mu,\mathcal X')<\e_0\right\}\mapsto\mathcal X',$$
where $\mathrm M(\S)$ is the space of signed measures on $\S$ equipped with $\mathrm{Lip}'$ topology.

We then apply Lemma \ref{lemma2} with $\e=\e_0$, hence for $L=L_\e$ we have

$$\psi\left(\frac{Ke^u}{\int_\S Ke^u}\right)\in\mathcal X'\quad\quad\quad\quad\quad\quad\forall u\in\mathcal J^{-L}.$$

Finally, to define $\Psi$ we need to pass from $\mathcal X'$ to $\mathcal X$, which we will using the map $\Pi$ defined in Proposition \ref{propretract}: we set
$$\Psi(u):=\Pi\circ\psi\left(\frac{Ke^u}{\int_\S Ke^u}\right)$$

\

To get the homotopy between $\Phi\circ\Psi$ and the identity map on $\mathcal X$, we just let $\l$ go to $+\infty$ in the definition of $\Phi$, namely:

\begin{eqnarray*}
\mathcal X\times[0,1]&\overset H\longrightarrow&\mathcal X\\
(\xi,t)&\mapsto&\left\{\begin{array}{ll}\Psi\circ\Phi^\frac{\l_0}{1-t}=\Pi\circ\psi\left(\frac{Ke^{\phi^\frac{\l_0}{1-t}_\xi}}{\int_\S Ke^{\phi^\frac{\l_0}{1-t}_\xi}}\right)&\text{if }t<1,\vspace{.3cm}\\\xi&\text{if }t=1.\end{array}\right.
\end{eqnarray*}

In fact, by the construction of $\Phi$, one has $\frac{Ke^{\phi^\l_\xi}}{\int_\S Ke^{\phi^\l_\xi}}\underset{\l\to+\infty}\weakto\xi$; moreover, $\psi$ being a retraction one has $\psi(\mu_n)\underset{n\to+\infty}\to\mu$ for any sequence of measures satisfying $\mu_n\underset{n\to+\infty}\weakto\mu$. Therefore, since $\Pi$ is also a retraction, we get the continuity of $H$ at $t=1$, which concludes the proof.

\end{proof}

\

{\bf Acknowledgements:} This paper is a part of the project ``The prescribed Gaussian and geodesic curvatures problem'' funded by Mathematisches Forschungsinstitut Oberwolfach as an Oberwolfach Leibniz Fellow. R.L-S. is currently funded under Juan de la Cierva Formaci\'on fellowship (FJCI-2017-33758) supported by the Ministry of Science, Innovation and Universities.

L.B. wants to express his gratitute to the Mathematisches Forschungsinstitut Oberwolfach for the kind hospitality received during his visit in October 2018.

R.L-S. wants to express his gratitude to the Mathematics and Physics Department of Roma Tre University for the kind hospitality received during his visit in June 2019.

Both authors wish to thank A. Jevnikar and D. Ruiz for several discussions and suggestions which has been of great help in the elaboration of Section 2.

\

\end{document}